\documentclass[a4paper]{amsart}

\usepackage{amscd,amsthm,amsfonts,amssymb,amsmath}
\usepackage{fullpage}
\usepackage[all]{xy}

    \newcommand{\BC}{{\mathbb {C}}} 
     \newcommand{\BF}{{\mathbb {F}}}

    \newcommand{\BQ}{{\mathbb {Q}}} \newcommand{\BR}{{\mathbb {R}}}
     \newcommand{\BT}{{\mathbb {T}}}

     \newcommand{\BZ}{{\mathbb {Z}}}

    \newcommand{\CC}{{\mathcal {C}}} 
    \newcommand{\CE}{{\mathcal {E}}} 
     \newcommand{\CH}{{\mathcal {H}}}
     \newcommand{\CJ}{{\mathcal {J}}}
     
    \newcommand{\CM}{{\mathcal {M}}} 
    \newcommand{\CO}{{\mathcal {O}}} \newcommand{\CP}{{\mathcal {P}}}
     \newcommand{\CR}{{\mathcal {R}}}
     \newcommand{\CT}{{\mathcal {T}}}

    \newcommand{\fa}{{\mathfrak{a}}} \newcommand{\fb}{{\mathfrak{b}}}
     
     \newcommand{\ff}{{\mathfrak{f}}}

    \newcommand{\fm}{{\mathfrak{m}}} \newcommand{\fn}{{\mathfrak{n}}}
     
    \newcommand{\fq}{{\mathfrak{q}}}

      \newcommand{\fB}{{\mathfrak{B}}}

     \newcommand{\fP}{{\mathfrak{P}}}

     \newcommand{\fX}{{\mathfrak{X}}}

    \newcommand{\Ann}{{\mathrm{Ann}}}

    \newcommand{\codim}{{\mathrm{codim}}}
    \newcommand{\coker}{{\mathrm{coker}}}

    \newcommand{\Div}{{\mathrm{Div}}} 
    \newcommand{\End}{{\mathrm{End}}}

    \newcommand{\Frob}{{\mathrm{Frob}}}

    \newcommand{\Gal}{{\mathrm{Gal}}} \newcommand{\GL}{{\mathrm{GL}}}
    
    \newcommand{\Hom}{{\mathrm{Hom}}}
    
    \newcommand{\id}{{\mathrm{id}}}\renewcommand{\Im}{{\mathrm{Im}}}

    \newcommand{\Res}{{\mathrm{Res}}}

    \newcommand{\SL}{{\mathrm{SL}}}

    \newcommand{\tr}{{\mathrm{tr}}}

    \theoremstyle{plain}
    \newtheorem{thm}{Theorem}[section] \newtheorem{cor}[thm]{Corollary}
    \newtheorem{lem}[thm]{Lemma}  \newtheorem{prop}[thm]{Proposition}
     \newtheorem{defn}[thm]{Definition}

\theoremstyle{remark} \newtheorem{remark}[thm]{Remark}
\theoremstyle{remark} 
\theoremstyle{remark} 

    \newcommand{\Poincare}{Poincar\'{e}~}

    \numberwithin{equation}{section}

\begin{document}

\title{Maximal Eisenstein ideals and cuspidal subgroups of modular Jacobian varieties}
\author{Yuan Ren}

\address{School of Mathematical Science, Sichuan Normal University, Chengdu, Sichuan, China}
\email{rytgyx@126.com}

\begin{abstract}
In this paper, we study the torsion subgroup of $J_0(N)$ over the field generated by those points in the cuspidal group, where $N$ is an odd positive integer. We prove that, considered as Hecke modules, this group and the cuspidal subgroup are both supported at the maximal Eisenstein ideals.
\end{abstract}

\maketitle

\tableofcontents
\section{Introduction}
For any positive integer $N$, let $X_0(N)_{/\BQ}$ be the the modular curve of level $\Gamma_0(N)$, and denote by $J_0(N)_{/\BQ}$ to be its Jacobian variety. Denote $\CC_N$ to be the subgroup of $J_0(N)(\bar{\BQ})$ generated by the cusps of $X_0(N)$, which will be called as \emph{the cuspidal subgroup} of $J_0(N)$. It is well known that any cusp of $X_0(N)$ can be represented as $[\frac{x}{d}]$ with $d$ is a positive divisor of $N$ and $x$ prime to $d$, and such a cusp is defined exactly over the cyclotomic field $\BQ(\mu_{(d,N/d)})$ (see \S1.3 of \cite{St}). It follows that $\CC_N\subseteq J_0(N)(\BQ(\mu_{\widetilde{N}}))$, where $\widetilde{N}=\prod p^{[\frac{1}{2}v_p(N)]}$ and therefore $\BQ(\mu_{\widetilde{N}})$ is the field generated by those points in the cuspidal group. Let $\CJ_N:=J_0(N)(\BQ_{\widetilde{N}})_\text{tor}$. Thus by the theorem of Manin-Drinfeld we have $\CC_N\subseteq\CJ_N$ .

Let $\CC_N(\BQ)=\CC^{G_{\BQ}}_N$ and $\CJ(\BQ)=\CJ^{G_{\BQ}}_N$ where $G_\BQ=\Gal(\bar{\BQ}/\BQ)$. A lot of work have already been devoted to the so-called generalized Ogg's conjecture, which claims that $\CJ_N(\BQ)=\CC_N(\BQ)$ for any positive integer $N$. We know that
\begin{itemize}
  \item $\CJ_p=\CC_p$ for any prime $p$ (see\cite{M});
  \item $\CJ_{p^n}(\BQ)\otimes_{\BZ}\BZ[{1}/{6}]=\CC_{p^n}\otimes_{\BZ}\BZ[{1}/{6}]$ for any prime $p$ and integer $n\geq1$(see \cite{L});
  \item $\CJ_D\otimes_{\BZ}\BZ[{1}/{6}]=\CC_D\otimes_{\BZ}\BZ[{1}/{6}]$ for any square-free $D$ (see \cite{Oh})
\end{itemize}
Note that all point in the cuspidal subgroup are $\BQ$-rational in the cases considered by Mazur and Ohta. However we know little about the whole group $\CJ_N$. In particular, we are wondering about the role played by $\CC_N$ in $\CJ_N$. In this paper, we study the support of $\CC_N$ and $\CJ_N$ as $\BT_0(N)$-modules. Here $\BT_0(N)$ is the full Hecke algebra whose definition will be recalled later. We will show that, away from $6N$, these two Hecke modules are both supported at the maximal Eisenstein ideals (see Definition~\ref{def-Eisenstein}). More precisely, we will prove the following

\begin{thm}\label{main}
Let $N$ be an odd positive integer and $\fm\subseteq\BT_0(N)$ be a maximal ideal. If $(\fm,6N)=1$, then the followings are equivalent:
\begin{enumerate}
  \item $\CC_N[\fm]$ is non-zero;
  \item $\CJ_N[\fm]$ is non-zero;
  \item $\fm$ is a maximal Eisenstein ideal.
\end{enumerate}
\end{thm}

There are two main ingredients in our proof. The first is the work of Stevens about relation between weight two-Eisenstein series and cuspidal subgroups which will be recalled in \S2. The second ingredient is the factorizations of Hecke algebras into various old- and new-quotients which will be studied in \S3. Then in \S4 we give the definition of maximal Eisenstein ideal and derive some of their basic properties. In \S5 we associate an Eisenstein series to any such a maximal Eisenstein ideal of residue characteristic prime to $6N$. The cuspidal subgroup associated to this Eisenstein series will enable us to prove Theorem~\ref{main} in the last section.
\\

\textbf{Notations}:
Let $\CH=\{z\in\BC|\Im(z)>0\}$ be the \Poincare upper half-plane. Let $\fq:\CH\rightarrow\BC,z\mapsto e^{2\pi iz},$ be the function on $\CH$ which will be used in the Fourier expansions of modular forms.

For any function $g$ on the upper half plane and any $\gamma=\left(
                        \begin{array}{cc}
                          a & b \\
                          c & d \\
                        \end{array}
                      \right)
\in\GL^+_2(\BR)$, we denote by $g|\gamma$ to be the function $z\mapsto\det(\gamma)\cdot(cz+d)^{-2}\cdot g(\gamma z)$, where $\gamma z=\frac{az+b}{cz+d}$.

\section{Preliminaries}
In this section, we recall the relation between weight two Eisenstein series and cuspidal subgroups of modular Jacobian varieties. For more details, the reader is referred to \cite{St} and \cite{St2}.

Fix a positive integer $N$, and denote by $\Gamma$ to be either $\Gamma_0(N)$ or $\Gamma_1(N)$. Let $\CM_2(\Gamma,\BC)$ be the space of weight two modular forms of level $\Gamma$, then we have the following decomposition
\[\CM_2(\Gamma,\BC)=S_2(\Gamma,\BC)\oplus\CE_2(\Gamma,\BC),\]
where $S_2(\Gamma,\BC)$ and $\CE_2(\Gamma,\BC)$ is the sub-space of cusp forms and Eisenstein series respectively. For any positive integer $n$, there is a Hecke operator $\CT^{\Gamma}_n$ acting on $\CM_2(\Gamma,\BC)$ preserving the above decomposition. We denote the restriction of $\CT^{\Gamma}_n$ to $S_2(\Gamma,\BC)$ by $T^{\Gamma}_n$. Let $\CT_\Gamma$ be the sub-$\BZ$-algebra of $\End(\CM_2(\Gamma,\BC))$ which is generated by $\{\CT^{\Gamma}_n\}_{n\geq1}$. Let $\BT_\Gamma$ be the $\BZ$-algebra generated by $\{T^{\Gamma}_n\}_{n\geq1}$, which is the restriction of $\CT_\Gamma$ to $\End(S_2(\Gamma,\BC))$. We call $\BT_\Gamma$ as the \emph{full} Hecke algebra of level $\Gamma$. If $\Gamma=\Gamma_0(N)$, then we denote $\BT_{\Gamma_0(N)}$ simply as $\BT_0(N)$, which is in fact generated by the $T^{\Gamma_0(N)}_\ell$ for all the primes $\ell$.

Let $X_\Gamma$ be the modular curve over $\BQ$ of level $\Gamma$. We denote by $cusp(\Gamma)$ to be the set of cusps of $X_\Gamma$, and by $Y_\Gamma$ to be the complement of $cusp(\Gamma)$ in $X_\Gamma$. Let $J_\Gamma$ be the Jacobian variety of $X_\Gamma$ over $\BQ$. For any $g\in\CM_2(\Gamma,\BC)$, let $\omega_g$ be the meromorphic differential on $X_\Gamma(\BC)$ whose pullback to the $\CH$ equals $g(z)dz$. Then the differential $\omega_g$ has all its poles supported at the cusps of $X_\Gamma$, and $g$ is a cusp form if and only if $\omega_g$ is everywhere holomorphic on $X_\Gamma$. Let $\Div^0(cusp(\Gamma);\BC)=\Div^0(cusp(\Gamma);\BZ)\otimes_\BZ\BC$. We define the following homomorphism of $\BC$-vector spaces
\[\delta_\Gamma:\CE_2({\Gamma,\BC})\rightarrow \Div^0(cusp(\Gamma);\BC),\]
such that
\[E\mapsto 2\pi i\sum_{x\in{cusp(\Gamma)}}\Res_x(\omega_E)\cdot[x],\]
where $\Res_x(\omega_E)$ is the residue of $\omega_E$ at $[x]$, so that $2\pi i\cdot\Res_x(\omega_E)=e_x\cdot a_0(E;[x])$, with $e_x$ the ramification index of $X_\Gamma$ at $x$ and $a_0(g;[x])$ the constant term of the Fourier expansion of $g$ at the cusp $x$. The homomorphism $\delta_\Gamma$ is actually an isomorphism by the theorem of Manin-Drinfeld. Because the restriction of $\omega_E$ to $Y_\Gamma$ is holomorphic, it induces the following periods integral homomorphism
\begin{align*}
\xi_E:H_1(Y_\Gamma(\BC),\BZ)\rightarrow\BC,\ [c]\mapsto\int_{c}\omega_E
\end{align*}
where $[c]$ is the homology class represented by a $1$-cycle $c$ on $Y_\Gamma(\BC)$. Note that, for any cusp $x$, we have
\begin{align*}
  \int_{c_x}\omega=2\pi i\cdot \Res_x(\omega_E),
\end{align*}
where $c_x$ is a small circle around $x$.

\begin{defn}\label{def}Let $E\in\CE_2(\Gamma,\BC)$ be a weight-two Eisenstein series of level $\Gamma$. We denote by $\CR_\Gamma(E)$ to be the sub-$\BZ$-module of $\BC$ generated by the coefficients of $\delta_\Gamma(E)$, and by $\CR(E)^{\vee}$ to be the dual $\BZ$-module of $\CR(E)$. Then :

$(1)$ the cuspidal subgroup $C_\Gamma(E)$ associated with $E$ is defined to be the subgroup of $J_\Gamma(\overline{\BQ})$ which is generated by $\{w_\Gamma\left(\phi\circ\delta_\Gamma(E)\right)\}_{\phi\in\CR(E)^{\vee}}$, where $w_\Gamma$ is the Atkin-Lehner involution;

$(2)$ the periods $\CP_\Gamma(E)$ of $E$ is defined to be the image of $\xi_E$. Since $\CP_\Gamma(E)$ contains $\CR_\Gamma(E)$ as we have seen above, we can define $A_\Gamma(E)$ to be the quotient $\CP_\Gamma(E)/\CR_\Gamma(E)$.
\end{defn}

\begin{remark}
The above definition of $C_\Gamma(E)$ is slightly different from that given in \cite{St} by adding an action of the Atkin-Lehner operator $w_\Gamma$. Since $w_\Gamma$ is an isomorphism, this modification does not change the order of the associated cuspidal subgroups. However, $C_\Gamma(E)$ is now annihilated by $I_\Gamma(E)$ under the usual action of the Hecke algebra, because $\CT^t_\ell\circ\delta_\Gamma=\delta_\Gamma\circ\CT_\ell$ and $\CT^t_\ell\circ w_\Gamma=w_\Gamma\circ\CT_\ell$ for any prime $\ell$.
\end{remark}

By Proposition 1.1 and Theorem 1.2 of \cite{St2}, $A_\Gamma(E)$ is a finite abelian group and there is a perfect pairing $C_\Gamma(E)\times A_\Gamma(E)\rightarrow\BQ/\BZ$. Thus, the determination of the order of $C_\Gamma(E)$ is reduced to that of $\CP_\Gamma(E)$. Below we briefly review a method due to Stevens for the computation of the periods. The reader is referred to \cite{St2} for details.

We first consider the case when $\Gamma=\Gamma_1(N)$. Denote by $S_N$ to be the set of all primes $p$ satisfying $p\equiv-1\pmod{4N}$. Let $\fX_N$ be the set of all non-quadratic Dirichlet character $\eta$ whose conductor is a prime in $S_N$, and let $\fX^{\infty}_N$ be the set of all non-quadratic Dirichlet character $\eta$ whose conductor is of the form $p^M_\eta$ with $p_\eta\in S_N$ and $M\in\BZ_{\geq1}$.

For any $E=\sum^{\infty}_{n=0}a_n(E;[\infty])\cdot \fq^n\in\CE_2(\Gamma_1(N),\BC)$ and any Dirichlet character $\eta$, the $L$-function associated to the pair $(E,\eta)$ is defined as
\begin{align*}
L(E,\eta,s):=\sum^{\infty}_{n=1}\frac{a_n(E;[\infty])\cdot\eta(n)}{n^s}.
\end{align*}
If $\eta\in\fX^{\infty}_N$ is of conductor $p^M_\eta$, then we define
\begin{align*}
  &\Lambda(E,\eta,1):=\frac{\tau(\overline{\eta})\cdot L(E,\eta,1)}{2\pi i},\\
  &\Lambda_{\pm}(E,\eta,1):=\frac{1}{2}(\Lambda(E,\eta,1)\pm\Lambda(E,\eta\cdot(\frac{}{p_\eta}),1)),
\end{align*}
where $(\frac{}{p_\eta})$ is the Legendre symbol associated to $p_\eta$. It is proved in Theorem 1.3 of \cite{St2} that, if $\CM$ is a finitely generated sub-$\BZ$-module of $\BC$, then the following are equivalent:

(St1) $\CP_{\Gamma_1(N)}(E)\subseteq\CM$;

(St2) $\CR_{\Gamma_1(N)}(E)\subseteq\CM$ and $\Lambda_{\pm}(E,\eta,1)\in\CM[\eta,\frac{1}{p_\eta}]$ for any $\eta\in\fX_{N}$;

(St3) $\CR_{\Gamma_1(N)}(E)\subseteq\CM$ and $\Lambda_{\pm}(E,\eta,1)\in\CM[\eta,\frac{1}{p_\eta}]$ for any $\eta\in\fX^{\infty}_N$.

Because $\Lambda_{\pm}(E,\eta,1)$ is essentially the Bernoulli numbers whose integrality and divisibility are well known (see Theorem 4.2 of \cite{St2}), we can then use the above result to determine the periods $\CP_{\Gamma_1(N)}(E)$ of $E$ and hence the order of $C_{\Gamma_1(N)}(E)$.

On the other hand, if $\Gamma=\Gamma_0(N)$, then Stevens' method can only determine $C_{\Gamma_0(N)}(E)$ up to its intersection with the Shimura subgroup. Recall that, if we denote by $\pi_N$ to be the natural projection of $X_1(N)$ to $X_0(N)$, then the Shimura subgroup of $J_0(N)$ is defined to be
\[\Sigma_N:=\ker\left(\pi^*_N:\ J_0(N)\rightarrow J_1(N)\right),\]
which is finite and of multiplicative type as a $G_\BQ$-module. For any $E\in\CE_2(\Gamma_0(N),\BC)$, we define
\[A^{(s)}_{\Gamma_0(N)}(E):=\left(\CP_{\Gamma_1(N)}(E)+\CR_{\Gamma_0(N)}(E)\right)/\CR_{\Gamma_0(N)}(E),\]
then it can be shown that there is an exact sequence
\[
\xymatrix@C=0.5cm{
  0 \ar[r] & \Sigma_N\bigcap C_{\Gamma_0(N)}(E) \ar[r] & C_{\Gamma_0(N)}(E) \ar[r] & A^{(s)}_{\Gamma_0(N)}(E) \ar[r] & 0, }
  \]
which enables us to determine the order of $C_{\Gamma_0(N)}(E)/\left(\Sigma_N\bigcap C_{\Gamma_0(N)}(E)\right)$.

At the end of this section, we recall some basic properties of the collection of functions $\{\phi_{\underline{x}}\}_{\underline{x}\in(\BQ/\BZ)^{\oplus2}}$ due to Hecke (see \cite{St}, Chapter 2, \S2.4) which we will need later. For any $\underline{x}=(x_1,x_2)\in(\BQ/\BZ)^{\oplus2}$, the Fourier expansion of $\phi_{\underline{x}}$ at infinity is
\begin{align}\label{3.2}
  \phi_{\underline{x}}(z)+\delta(\underline{x})\cdot\frac{i}{2\pi(z-\overline{z})}=\frac{1}{2} B_2(x_1)-P_{\underline{x}}(z)-P_{-\underline{x}}(z)
\end{align}
for any $z\in\CH$, where $B_2(t)=\langle{t}\rangle^2-\langle{t}\rangle+\frac{1}{6}$ is the second Bernoulli polynomial and
\begin{align}\label{3.3}
P_{\underline{x}}(z)=\sum_{k\in\BQ_{>0},k\equiv x_1(1)}k\sum^{\infty}_{m=1}e^{2\pi im(kz+x_2)},
\end{align}
and $\delta(\underline{x})$ is defined to be $1$ or $0$ according to $\underline{x}=0$ or not. If $\underline{x}\neq0$, then $\phi_{\underline{x}}$ is a (holomorphic) Eisenstein series. The whole collection of functions satisfy the following important \emph{distribution law}
\begin{align}\label{3.4}
  \phi_{\underline{x}}=\sum_{\underline{y}:\ \underline{y}\cdot \alpha=\underline{x}}\phi_{\underline{y}}|\alpha,
\end{align}
where $\alpha\in M_2(\BZ)$ with $\det(\alpha)>0$. In particular, we have $\phi_{\underline{x}}|\gamma=\phi_{\underline{x}\cdot\gamma}$ for any $\gamma\in\SL_2(\BZ)$.
\\

\section{Factorizations of Hecke algebras}
\textbf{3.1. }Let $N$ be a positive integer. If $N'$ and $d$ are positive divisors of $N$ such that $N'd\mid N$, then there is a degeneracy morphism
\[\pi^N_{N',d}: X_0(N)\rightarrow X_0(N'),\ ([E,C])\mapsto[E/C[d],C[N'd]/C[d]]\]
where $E$ is an elliptic curve and $C$ a cyclic subgroup of order $N$. If $N''$ and $d'$ are positive divisors of $N'$ such that $N''d'\mid N'$, then it is easy to see that we have the following equation
\begin{align}
  \pi^N_{N'',dd'}=\pi^{N'}_{N'',d'}\circ\pi^N_{N',d}.
\end{align}
Moreover if $[E,C]=[\BC/\BZ z+\BZ,\langle\frac{1}{N}\rangle]$ for some $z\in\CH$, then
\begin{align*}
\pi^N_{N',d}([E,C])&=\pi^N_{N',d}([\BC/\BZ z+\BZ,\langle\frac{1}{N}\rangle])\\
&=[\BC/\BZ z+\frac{1}{d}\BZ,\langle\frac{1}{N'd}\rangle]\\
&=[\BC/\BZ(dz)+\BZ,\langle\frac{1}{N'}\rangle].
\end{align*}
Thus the morphism $\pi^N_{N',d}$ can be analytically described as sending $z$ to $dz$ for any $z\in\CH$. By the Picard functoriality that, there is a homomorphism
\[\pi^{N*}_{N',d}:J_0(N')\rightarrow J_0(N)\]
between abelian varieties over $\BQ$.

\begin{defn}\label{p-old}Let $N$ be a positive integer and $p\mid N$ be a prime. Denote
\[\iota_p:=\pi^{N*}_{N/p,1}+\pi^{N*}_{N/p,p}:\ J_0(N/p)^2\rightarrow J_0(N).\]
Then we define
\[J_0(N)_\text{p-old}:=\Im(\iota_p),\]
which is called as the \textbf{p-old subvariety} of $J_0(N)$, and define
\[J_0(N)^\text{p-new}:=\coker(\iota_p),\]
which is called as the \textbf{p-new quotient variety} of $J_0(N)$.
\end{defn}

\textbf{3.2. }For our later use, here we take a short digression to recall the algebraic definition of the Hecke operators. Fix $N\in\BZ_{\geq1}$ to be a positive integer. For any prime $\ell$, let $X_0(N,\ell)_{/\BQ}$ be the compactified coarse moduli space which classifies all triples $[E,C,D]$, where $E$ is an elliptic curve over some $\BQ$-scheme, $C$ a cyclic subgroup of order $N$, and $D$ a cyclic subgroup of order $\ell$ such that $D\bigcap C={0}$. Let
\[\alpha_\ell,\ \beta_\ell:X_0(N,\ell)\rightarrow X_0(N)\]
be morphisms between smooth curves over $\BQ$ such that
\begin{align*}
\begin{cases}
\alpha_\ell\left([E,C,D]\right)=[E,C]\\
\beta_\ell\left([E,C,D]\right)=[E/D,(C+D)/D],
\end{cases}
\end{align*}
then we define the \emph{$\ell$-th Hecke operator} on $J_0(N)$ as
\[T^{\small{\Gamma_0(N)}}_\ell:=\beta_{\ell*}\circ\alpha^*_\ell\in\End_{\BQ}(J_0(N)).\]
For any point $[E,C]$ on $X_0(N)$, we have the following equation of divisors
\begin{align}
T^{\small{\Gamma_0(N)}}_\ell([E,C])=\sum_{D}[E/D,(C+D)/D],
\end{align}
where $D$ runs over all cyclic subgroup of order $\ell$ such that $D\bigcap C={0}$. Define
\[\BT_0(N):=\BZ[\{T^{\small{\Gamma_0(N)}}_\ell\}_\ell]\subseteq\End_{\BQ}(J_0(N)),\]
which is called to be the (full) Hecke algebra of level $N$. Since
\begin{align*}
T^{\small{\Gamma_0(N)}}_{\ell_2}\circ T^{\small{\Gamma_0(N)}}_{\ell_1}([E,C])&=T^{\small{\Gamma_0(N)}}_{\ell_2}\left(\sum_{D_1}[E/{D_1},(C+D_1)/{D_1}]\right)\\
&=\sum_{D_1}\sum_{D_2}[E/(D_1+D_2),(C+D_1+D_2)/(D_1+D_2)]\\
&=T^{\small{\Gamma_0(N)}}_{\ell_1}\circ T^{\small{\Gamma_0(N)}}_{\ell_2}([E,C])
\end{align*}
for any two primes $\ell_1\neq\ell_2$, we find that $\BT_0(N)$ is a commutative ring, which is in fact free of finite rank over $\BZ$ (see \cite{Ri}). Moreover, for a generic complex point $[\BC/\BZ z+\BZ,\langle1/N\rangle]$ in $X_0(N)(\BC)$, it follows from the definition of $\alpha_\ell$ that
\begin{align*}
\alpha^{*}_\ell([\BC/\BZ z+\BZ,\langle\frac{1}{N}\rangle])=
\begin{cases}
\sum^{\ell-1}_{k=0}[\BC/\BZ z+\BZ,\langle\frac{1}{N}\rangle+\langle\frac{z+k}{\ell}\rangle]+[\BC/\BZ z+\BZ,\langle\frac{1}{N\ell}\rangle]&,\text{if}\ \ell\nmid N\\
\sum^{\ell-1}_{k=0}[\BC/\BZ z+\BZ,\langle\frac{1}{N}\rangle+\langle\frac{z+k}{\ell}\rangle] &,\text{if}\ \ell\mid N,
\end{cases}
\end{align*}
so that analytically we find that
\begin{align*}
T^{\Gamma_0(N)}_\ell=
\begin{cases}
\sum^{\ell-1}_{k=0}\left(
                     \begin{array}{cc}
                       1 & k \\
                       0 & \ell \\
                     \end{array}
                   \right)
+\left(
                                                \begin{array}{cc}
                                                  \ell & 0 \\
                                                  0 & 1 \\
                                                \end{array}
                                              \right)
&,\text{if}\ \ell\nmid N\\
\sum^{\ell-1}_{k=0}\left(
                     \begin{array}{cc}
                       1 & k \\
                       0 & \ell \\
                     \end{array}
                   \right)&,\text{if}\ \ell\mid N,
\end{cases}
\end{align*}
which therefore coincides with the classical expressions we discussed before.

\begin{lem}\label{lemma1}Let $N\in\BZ_{\geq1}$ and $p\mid N$ be a prime. If $\ell\neq p$ is a prime, then
\begin{align*}
\begin{cases}
T^{\Gamma_0(N)}_\ell\circ\pi^{N*}_{N/p,1}=\pi^{N*}_{N/p,1}\circ T^{\Gamma_0(N/p)}_\ell\\
T^{\Gamma_0(N)}_\ell\circ\pi^{N*}_{N/p,p}=\pi^{N*}_{N/p,p}\circ T^{\Gamma_0(N/p)}_\ell
\end{cases}
\end{align*}
\end{lem}
\begin{proof}
For any point $[E,C]$ on $Y_0(N/p)$, we decompose $C=\prod_{q}C_q$ with $C_q$ be the $q$-primary part for any prime $q$. Then
\begin{align*}
\pi^{N*}_{N/p,1}([E,C])&=\pi^{N*}_{N/p,1}([E,\prod_{q}C_q])\\
&=\sum[E,(\prod_{q\neq p}C_q)\times C'_p],
\end{align*}
where $C'_p$ runs over all cyclic subgroups of order $p^{v_p(N)}$ such that $C'_p[p^{v_p(N)-1}]=C_p$. Since $\ell\neq p$, the morphisms used to define the $\ell$-th Hecke operators have their affections only on the $\ell$-part and hence leave the $p$-part unchanged. So the first equation follows. The proof for $\pi^{N*}_{N/p,p}$ is similar.
\end{proof}

Let $p$ a prime divisor of $N$. Recall that we have the $p$-th Atkin-Lehner operator
\[w_p: X_0(N)\rightarrow X_0(N),\ ([E,C])\mapsto[E/C[p^v_p(N)],(C+E[p^{v_p(N)}])/C[p^{v_p(N)}]],\]
where $v_p(N)$ is the $p$-adic valuation of $N$. Since we have
\begin{align*}
w^2_p([E,C])&=w_p([E/C[p^{v_p(N)}],(C+E[p^{v_p(N)}])/C[p^{v_p(N)}]])\\
&=[E/E[p^{v_p(N)}],(C+\frac{1}{p^{v_p(N)}}C[p^{v_p(N)}])/E[p^{v_p(N)}]]\\
&=[E,C],
\end{align*}
it follows that $w^2_p=\id$, that is, $w_p$ is an involution on $X_0(N)$. In particular we have $w^*_p=W_{p,*}$ as automorphisms on $J_0(N)$ which will be simply denoted as $w_p$.

\begin{lem}\label{lemma2}Let $N\in\BZ_{\geq1}$ and $p$ be a prime such that $p\parallel N$. Then
\begin{align*}
\begin{cases}
\pi^{N}_{N/p,1}\circ w_p=\pi^{N}_{N/p,p}\\
\pi^{N}_{N/p,p}\circ w_p=\pi^{N}_{N/p,1},
\end{cases}
\end{align*}
and we have that
\[\pi^{N*}_{N/p,p}\circ\pi^{N}_{N/p,1*}=T^{\Gamma_0(N)}_p+w_p.\]
\end{lem}
\begin{proof}
For any $[E,C]$ in $Y_0(N)\subseteq X_0(N)$, we have
\begin{align*}
  \pi^{N}_{N/p,1}\circ w_p([E,C])&=\pi^{N}_{N/p,1}\left([E/C[p],(C+E[p])/C[p]]\right)\\
  &=[E/C[p],C/C[p]]\\
  &=\pi^{N}_{N/p,p}([E,C]),
\end{align*}
and similarly,
\begin{align*}
  \pi^{N}_{N/p,p}\circ w_p([E,C])&=\pi^{N}_{N/p,p}\left([E/C[p],(C+E[p])/C[p]]\right)\\
  &=[E/E[p],(C+E[p])/E[p]]\\
  &=[E,C[N/p]]=\pi^{N}_{N/p,1}([E,C]),
\end{align*}
which prove the first assertion. For the second assertion, note that we have
\begin{align*}
\pi^{N*}_{N/p,p}\circ\pi^{N}_{N/p,1*}([E,C])&=\pi^{N*}_{N/p,p}([E,C[N/p]])\\
&=\sum_{i}[E_i,C_i],
\end{align*}
where the sum runs over all points which are mapped to $[E,C[N/p]]$ by $\pi^{N}_{N/p,p}$. Because $\pi^{N}_{N/p,p}([E_i,C_i])=[E_i/C_i[p],C_i/C_i[p]]$, it follows that there is an isogeny $\phi_i:E_i\rightarrow E$ with $\ker(\phi_i)=C_i[p]$, which induces an isomorphism $C_i/C_i[p]\simeq C[N/p]$. Let $D_i=\phi_i(E_i[p])$. Then we find that there is an isomorphism $\psi_i:E_i\simeq E/D_i$, such that where $C_i[N/p]\subseteq E_i$, the image of $C_i/C_i[p]$ under $[p]$, is mapped to $\overline{C[N/p]}$ in $E/D_i$ via $\psi_i$. On the other hand, if we write $C_i[p]=\BZ/p\BZ\cdot{(x)}$ for some generator $x$, then the pre-image of $C_i[p]$ under $E_i\rightarrow E_i/E_i[p]\simeq E_i$ can be described as $\BZ/p^2\BZ\cdot{(\frac{x}{p})}+E_i[p]$. Thus through the identification $E_i/C_i[p]\simeq E$, we find that $\BZ/p^2\BZ\cdot{(\frac{x}{p})}+E_i[p]$ is mapped to $E[p]$. It follows that $\psi_i(C_i[p])=E[p]/D_i$, and therefore
\begin{align*}
\pi^{N*}_{N/p,p}\circ\pi^{N}_{N/p,1*}([E,C])&=\sum_{i}[E/D_i,(C[N/p]+E[p])/D_i]\\
&=\sum_{i}[E/D_i,(C+E[p])/D_i],
\end{align*}
where $D_i$ runs over all cyclic subgroup of order $p$. We have thus proved the lemma because
\[T^{\Gamma_0(N)}([E,C])=\sum_{D\bigcap C=0}[E/D,(C+E[p])/D],\]
and
\[w_p([E,C])=[E/C[p],(C+E[p])/C[p]].\]
\end{proof}

\begin{cor}\label{cor1}
If $N\in\BZ_{\geq1}$ and $p$ is a prime such that $p\parallel N$, then
\begin{align*}
  T^{\Gamma_0(N)}_p\circ(\pi^{N*}_{N/p,1}\ \pi^{N*}_{N/p,p})=(\pi^{N*}_{N/p,1}\ \pi^{N*}_{N/p,p})\circ\left(
                                                                                                   \begin{array}{cc}
                                                                                                     0 & -1 \\
                                                                                                     p & {T^{\Gamma_0(N/p)}_p} \\
                                                                                                   \end{array}
                                                                                                 \right).
\end{align*}
Therefore both $J_0(N)_\text{p-old}$ and $J_0(N)^\text{p-new}$ are stable under the action of $\BT_0(N)$. In particular we have $T^{\Gamma_0(N)}_p=-w_p$ on $J_0(N)^\text{p-new}$.
\end{cor}
\begin{proof}
By Lemma~\ref{lemma1}, it suffices to show that $J_0(N)_\text{p-old}$ is stable under the action of the $p$-th Hecke operator $T^{\Gamma_0(N)}_p$. Since $\deg(\pi^{N}_{N/p,1})=p+1$, we find from Lemma~\ref{lemma2} that
\begin{align*}
  T^{\Gamma_0(N)}_p\circ\pi^{N*}_{N/p,1}&=\pi^{N*}_{N/p,p}\circ\pi^{N}_{N/p,1*}\circ\pi^{N*}_{N/p,1}-w_p\circ\pi^{N*}_{N/p,1}\\
  &=(p+1)\cdot\pi^{N*}_{N/p,p}-\pi^{N*}_{N/p,p}\\
  &=\pi^{N*}_{N/p,p},
\end{align*}
and similarly
\begin{align*}
  T^{\Gamma_0(N)}_p\circ\pi^{N*}_{N/p,p}&=\pi^{N*}_{N/p,p}\circ\pi^{N}_{N/p,1*}\circ\pi^{N*}_{N/p,p}-w_p\circ\pi^{N*}_{N/p,p}\\
  &=\pi^{N*}_{N/p,p}\circ T^{\Gamma_0(N/p)}-\pi^{N*}_{N/p,1}.
\end{align*}
Thus we have proved the first assertion, from which follows the stability of $J_0(N)_\text{p-old}$ and hence that of $J_0(N)^\text{p-new}$. The last assertion is clear from Lemma~\ref{lemma2}.
\end{proof}

\begin{lem}\label{lemma3}
If $p$ is a prime such that $p^2\mid N$, then we have
\begin{align*}
\begin{cases}
T^{\Gamma_0(N)}_p=\pi^{N*}_{N/p,p}\circ\pi^{N}_{N/p,1*}\\
T^{\Gamma_0(N/p)}_p=\pi^{N}_{N/p,1*}\circ\pi^{N*}_{N/p,p}.
\end{cases}
\end{align*}
\end{lem}
\begin{proof}
It suffices to verify the above equality after base changing to $\BC$. For any (generic) point $z\in\CH$, we have thpat
\[\pi^{N*}_{N/p,p}\left([\BC/\BZ\cdot z+\BZ,\langle\frac{1}{N/p}\rangle]\right)=\sum[\BC/\BZ\cdot z'+\BZ,\langle\frac{1}{N}\rangle],\]
where the sum on the right hand side runs over all those points which are mapped to $[\BC/\BZ\cdot z+\BZ,\langle\frac{1}{N/p}\rangle]$ by $\pi^{N}_{N/p,p}$. Thus, for any $z'$, there is some $\gamma\in\Gamma_0(N/p)$ such that $pz'=\gamma(z)$, and hence
\begin{align*}
z'=\left(
     \begin{array}{cc}
       p^{-1} & 0 \\
       0 & 1 \\
     \end{array}
   \right)\gamma\left(
            \begin{array}{cc}
              1 & -k \\
              0 & 1 \\
            \end{array}
          \right)\left(
                   \begin{array}{cc}
                     p & 0 \\
                     0 & 1 \\
                   \end{array}
                 \right)(\frac{z+k}{p}),
\end{align*}
where $k\in\{0,1,...,p-1\}$ is the unique integer such that
\[\left(
     \begin{array}{cc}
       p^{-1} & 0 \\
       0 & 1 \\
     \end{array}
   \right)\gamma\left(
            \begin{array}{cc}
              1 & -k \\
              0 & 1 \\
            \end{array}
          \right)\left(
                   \begin{array}{cc}
                     p & 0 \\
                     0 & 1 \\
                   \end{array}
                 \right)\in\Gamma_0(N).\]
It follows that $[\BC/\BZ\cdot z'+\BZ,\langle\frac{1}{N}\rangle]=[\BC/\BZ\cdot\frac{z+k}{p}+\BZ,\langle\frac{1}{N}\rangle]$, and therefore
\[\pi^{N*}_{N/p,p}=\sum^{p-1}_{k=0}\left(
                                     \begin{array}{cc}
                                       1 & k \\
                                       0 & p \\
                                     \end{array}
                                   \right),
\]
which implies that both $\pi^{N}_{N/p,1*}\circ\pi^{N*}_{N/p,p}$ and $\pi^{N*}_{N/p,p}\circ\pi^{N}_{N/p,1*}$ are analytically given as $\sum^{p-1}_{k=0}\left(
                                                                                                                                      \begin{array}{cc}
                                                                                                                                        1 & k \\
                                                                                                                                        0 & p \\
                                                                                                                                      \end{array}
                                                                                                                                    \right)
$ and hence prove the lemma.
\end{proof}

\begin{cor}\label{cor2}
If $p$ is a prime such that $p^2\mid N$, then we have
\begin{align*}
  T^{\Gamma_0(N)}_p\circ(\pi^{N*}_{N/p,1}\ \pi^{N*}_{N/p,p})=(\pi^{N*}_{N/p,1}\ \pi^{N*}_{N/p,p})\circ\left(
                                                                                                   \begin{array}{cc}
                                                                                                     0 & 0 \\
                                                                                                     p & {T^{\Gamma_0(N/p)}_p} \\
                                                                                                   \end{array}
                                                                                                 \right).
\end{align*}
Therefore both $J_0(N)_\text{p-old}$ and $J_0(N)^\text{p-new}$ are stable under the action of $\BT_0(N)$. In particular we have $T^{\Gamma_0(N)}_p=0$ on $J_0(N)^\text{p-new}$.
\end{cor}
\begin{proof}
Since $p\mid N/p$, we find that $\deg(\pi^{N}_{N/p,1})=p$. Then it follows from Lemma~\ref{lemma3} that
\begin{align*}
T^{\Gamma_0(N)}_p\circ\pi^{N*}_{N/p,1}&=\pi^{N*}_{N/p,p}\circ\pi^{N}_{N/p,1*}\circ\pi^{N*}_{N/p,1}\\
&=\pi^{N*}_{N/p,p}\circ[p].
\end{align*}
On the other hand, since
\begin{align*}
(T^{\Gamma_0(N)}_p)^2&=T^{\Gamma_0(N)}_p\circ\pi^{N*}_{N/p,p}\circ\pi^{N}_{N/p,1*}\\
&=\pi^{N*}_{N/p,p}\circ\pi^{N}_{N/p,1*}\circ\pi^{N*}_{N/p,p}\circ\pi^{N}_{N/p,1*}\\
&=\pi^{N*}_{N/p,p}\circ T^{\Gamma_0(N/p)}_p\circ\pi^{N}_{N/p,1*},
\end{align*}
we find that $T^{\Gamma_0(N)}_p\pi^{N*}_{N/p,p}\circ\pi^{N}_{N/p,1*}=\pi^{N*}_{N/p,p}\circ T^{\Gamma_0(N/p)}_p\circ\pi^{N}_{N/p,1*}$, and hence $T^{\Gamma_0(N)}_p\pi^{N*}_{N/p,p}=\pi^{N*}_{N/p,p}\circ T^{\Gamma_0(N/p)}_p$ because $\pi^{N}_{N/p,1*}$ is surjective. The other assertions are then clear from the definitions.
\end{proof}

\begin{defn}\label{def1}
Let $N$ be a positive integer. Then for any prime $p\mid N$ we define
\begin{align*}
\begin{cases}
\BT_0(N)^\text{p-old}:=\Im\left(\BT_0(N)\rightarrow\End_{\BQ}(J_0(N)_\text{p-old})\right)\\
\BT_0(N)^\text{p-new}:=\Im\left(\BT_0(N)\rightarrow\End_{\BQ}(J_0(N)^\text{p-new})\right),
\end{cases}
\end{align*}
which are called as \textbf{the p-old quotient} and \textbf{the p-new quotient} of $\BT_0(N)$ respectively. Thus there are two $\BZ$-algebra surjections
\begin{align*}
\begin{cases}
\BT_0(N)\twoheadrightarrow\BT_0(N)^\text{p-old}\\
\BT_0(N)\twoheadrightarrow\BT_0(N)^\text{p-new},
\end{cases}
\end{align*}
which combine to give an injection
\begin{align*}
\BT_0(N)\hookrightarrow\BT_0(N)^\text{p-old}\times\BT_0(N)^\text{p-new}.
\end{align*}
\end{defn}

\textbf{3.3. }The $\BZ$-algebras $\BT_0(N)^\text{p-old}$ and $\BT_0(N/p)$ are closely related. By Lemma~\ref{lemma1}, $\iota_p=\pi^{N*}_{N/p,1}+\pi^{N*}_{N/p,p}$ commutes with the $\ell$-th Hecke operators for any $\ell\neq p$. Let
\[R:=\BZ[\{T^{\Gamma_0(N/p)}_\ell\}_{\ell\neq p}].\]
Then $R$ acts on $J_0(N/p)^2$ diagonally, and hence on $J_0(N)_\text{p-old}$ via $\iota_p$. It follows that there is an injection $R\hookrightarrow\End_{\BQ}(J_0(N)_\text{p-old})$, and we have that
\[\BT_0(N)^\text{p-old}=R[T^{\Gamma_0(N)}_p].\]
By Corollary~\ref{cor1}, if $p$ is a prime with $p\| N$, then the image of the $p$-th Hecke operator in $\BT_0(N)^\text{p-old}$ satisfies the equation $x^2-T^{\Gamma_0(N/p)}x+p=0$, that is to say, we have
\begin{align}\label{2.3}
(T^{\Gamma_0(N)}_p)^2-T^{\Gamma_0(N/p)}\cdot T^{\Gamma_0(N)}_p+p=0
\end{align}
in $\BT_0(N)^\text{p-old}$. Moreover by the lemma on P495 in \cite{Wiles} we have $\BT_0(N/p)[\frac{1}{2}]=R[\frac{1}{2}]$, and $\BT_0(N/p)=R$ if $p\neq2$. So we have that
\begin{align}\label{2.4}
\BT_0(N)^\text{p-old}[\frac{1}{2}]\simeq\BT_0(N/p)[\frac{1}{2},x]/(x^2-T^{\Gamma_0(N/p)}x+p);
\end{align}
and if $p\neq2$, then we have that
\begin{align}\label{2.5}
\BT_0(N)^\text{p-old}\simeq\BT_0(N/p)[x]/(x^2-T^{\Gamma_0(N/p)}x+p).
\end{align}
\\

\textbf{3.4. }Hereafter we assume that $p^2\mid N$. Define for any $i=0,1,2,...,v_p(N)-1$
\begin{align*}
  J_0(N)^{(i)}_\text{p-old}:&=\pi^{N*}_{N/p^i,1}(J_0(N/p^i)_\text{p-old})+\pi^{N*}_{N/p,p}(J_0(N/p))\\
  &=\pi^{N*}_{N/p^{i+1},1}(J_0(N/p^{i+1}))+\pi^{N*}_{N/p,p}(J_0(N/p)).
\end{align*}
where $J_0(N)^{(-1)}_\text{p-old}:=J_0(N)$. It follows that there is a filtration
\begin{align*}
  J_0(N)\supseteq J_0(N)^{(0)}_\text{p-old}\supseteq J_0(N)^{(1)}_\text{p-old}\supseteq...\supseteq J_0(N)^{(n_p-1)}_\text{p-old},
\end{align*}
and we denote the sub-quotients as
\[J_0(N)^\text{p-new}_{(i)}:=J_0(N)^{(i-1)}_\text{p-old}/J_0(N)^{(i)}_\text{p-old}\]
for any $i=0,...,v_p(N)-1$. 

\begin{lem}\label{lemma4}
For any $i=0,1,...,n_p-1$, both $J_0(N)^{(i)}_\text{p-old}$ and $J_0(N)^{(i)}_\text{p-new}$ are stable under the action of $\BT_0(N)$.
\end{lem}
\begin{proof}
It suffices to prove that $J_0(N)^{(i)}_\text{p-old}$ is stable. By Lemma~\ref{lemma1} we find that both $\pi^{N*}_{N/p^{i+1},1}=\pi^{N*}_{N/p,1}\circ...\circ\pi^{N/p^i*}_{N/p^{i+1},1}$ and $\pi^{N*}_{N/p,p}$ commute with the $\ell$-the Hecke operators for any prime $\ell\neq p$. Since we have that
\[T^{\Gamma_0(N)}_p\circ\pi^{N*}_{N/p,p}=\pi^{N*}_{N/p,p}\circ T^{\Gamma_0(N/p)}_p\]
and that
\begin{align*}
  T^{\Gamma_0(N)}\circ\pi^{N*}_{N/p^{i+1},1}&=T^{\Gamma_0(N)}\circ\pi^{N*}_{N/p,1}\circ...\circ\pi^{N/p^i*}_{N/p^{i+1},1}\\
  &=p\cdot\pi^{N*}_{N/p,p}\circ...\circ\pi^{N/p^i*}_{N/p^{i+1},1},
\end{align*}
by Corollary~\ref{cor2}, it follows that $T^{\Gamma_0(N)}_p(J_0(N)^{(i)}_\text{p-old})\subseteq\pi^{N*}_{N/p,p}(J_0(N/p))\subseteq J_0(N)^{(i)}_\text{p-old}$, which completes the proof.
\end{proof}

Let $n_p=v_p(N)$ and we \textbf{assume that $p\neq2$} so that  $T^{\Gamma_0(N/p^{n_p})}_p\in\BT_0(N/p^{n_p-1})^\text{p-old}$. Let
\begin{align*}
\epsilon_p:=T^{\Gamma_0(N/p^{n_p-1})}_p-T^{\Gamma_0(N/p^{n_p})}_p;
\end{align*}
and for any $i=1,...,n_p-1$, we let
\begin{align*}
\tau_{p,i}:=\pi^{N*}_{N/p^{n_p-1},p^{i-1}}+\pi^{N*}_{N/p^{n_p-1},p^{i}}\circ\epsilon_p.
\end{align*}
Then we define
\begin{align*}
J_0(N)^{(n_p)}_\text{p-old}:=\sum^{n_p-1}_{i=1}\tau_{p,i}(J_0(N/p^{n_p-1})^2_\text{p-old}).
\end{align*}

\begin{lem}\label{lemma6}
$J_0(N)^{(n_p)}_\text{p-old}$ is stable under the action of $\BT_0(N)$. Moreover, $(T^{\Gamma_0(N)}_p)^{n_p-1}$ acts as zero on $J_0(N)^{(n_p)}_\text{p-old}$.
\end{lem}
\begin{proof}
By Lemma~\ref{lemma1}, $J_0(N)^\text{p-old}_{(n_p)}$ is stable under the action of $\ell$-th Hecke operator for any prime $\ell\neq p$. If $i\leq n_p-2$, then we have
\begin{align*}
T^{\Gamma_0(N)}_p\circ\pi^{N*}_{N/p^{n_p-1},p^i}=&T^{\Gamma_0(N)}_p\circ\pi^{N*}_{N/p,1}\circ\pi^{N/p*}_{N/p^{n_p-1},p^{i}}\\
=&p\cdot\pi^{N*}_{N/p,p}\circ\pi^{N/p*}_{N/p^{n_p-1},p^{i}}\\
=&p\cdot\pi^{N*}_{N/p^{n_p-1},p^{i+1}},
\end{align*}
so that
\begin{align*}
T^{\Gamma_0(N)}_p\circ\tau_{p,i}&=T^{\Gamma_0(N)}_p\circ(\pi^{N*}_{N/p^{n_p-1},p^{i-1}}+\pi^{N*}_{N/p^{n_p-1},p^{i}}\circ\epsilon_p)\\
&=p\cdot(\pi^{N*}_{N/p^{n_p-1},p^{i}}+\pi^{N*}_{N/p^{n_p-1},p^{i+1}}\circ\epsilon_p)\\
&=p\cdot\tau_{p,i+1}.
\end{align*}
Moreover, since
\begin{align*}
T^{\Gamma_0(N)}_p\circ\tau_{p,n_p-1}&=T^{\Gamma_0(N)}_p\circ(\pi^{N*}_{N/p^{n_p-1},p^{n_p-2}}+\pi^{N*}_{N/p^{n_p-1},p^{n_p-1}}\circ\epsilon_p)\\
&=\pi^{N*}_{N/p^{n_p-1},p^{n_p-1}}\circ(p+T^{\Gamma_0(N/p^{n_p-1})}_p\circ\epsilon_p)\\
&=0,
\end{align*}
we find that $J_0(N)^\text{p-old}_{(n_p)}$ is also stable under the action of $T^{\Gamma_0(N)}_p$ and hence under the action of $\BT_0(N)$, and by the way also  prove the second assertion.
\end{proof}

\begin{defn}\label{hecke old and new}
Let $N$ be a positive integer. If $p$ is a prime with $p$ a prime with $n_p=v_p(N)\geq2$, then we define
\begin{align}\label{def-old and new}
\begin{cases}
\BT_0(N)^\text{p-old}_{(i)}:=\Im(\BT_0(N)\rightarrow\End_\BQ(J_0(N)^{(i)}_{\text{p-old}}))\\
\BT_0(N)^\text{p-new}_{(i)}:=\Im(\BT_0(N)\rightarrow\End_\BQ(J_0(N)^{\text{p-new}}_{(i)}))
\end{cases}
\end{align}
for any $i=0,...,v_p(N)-1$, and define
\begin{align}
  \BT_0(N)^\text{p-old}_{(n_p)}:=\Im(\BT_0(N)\rightarrow\End_\BQ(J_0(N)^{(n_p)}_\text{p-old})).
\end{align}
Note that $\BT_0(N)^\text{p-old}_{(0)}=\BT_0(N)^\text{p-old}$ and $\BT_0(N)^\text{p-new}_{(0)}=\BT_0(N)^\text{p-new}$. 
\end{defn}

\begin{prop}\label{decomposition-prop}
Let $N$ be a positive integer and $p$ an odd prime with $n_p=v_p(N)\geq2$. Then there are $\BZ$-algebra injections
\begin{align*}
\begin{cases}
  \BT_0(N)\hookrightarrow\BT_0(N)^\text{p-new}\times\BT_0(N)^\text{p-old}\\
  \BT_0(N)^\text{p-old}\hookrightarrow(\prod^{n_p-1}_{i=1}\BT_0(N)^\text{p-new}_{(i)})\times\BT_0(N)^\text{p-old}_{(n_p)}\times\BT_0(N/p),
\end{cases}
\end{align*}
such that
\begin{enumerate}
  \item For any $i=1,...,n_p-1$, we have $T^{\Gamma_0(N)}_p\mapsto0$ under $\BT_0(N)\twoheadrightarrow\BT_0(N)^\text{p-new}_{(i)}$. Moreover , $\BT_0(N)^\text{p-new}_{(i)}$ is a quotient of $\BT_0(N/p^i)^\text{p-new}$;
  \item Under $\BT_0(N)^\text{p-old}\twoheadrightarrow\BT_0(N/p)$, we have $T^{\Gamma_0(N)}_\ell\mapsto T^{\Gamma_0(N/p)}_\ell$ for any prime $\ell$;
  \item Under $\BT_0(N)^\text{p-old}\twoheadrightarrow\BT_0(N)^\text{p-old}_{(n_p)}$, we have $T^{\Gamma_0(N)}_\ell\mapsto0$ for any prime $\ell\neq p$ and $(T^{\Gamma_0(N)}_p)^{n_p-1}\mapsto0$.
\end{enumerate}
\end{prop}
\begin{proof}
We first note that if $i\geq1$, then there are surjective $\BZ$-algebra homomorphisms $\BT_0(N)\twoheadrightarrow\BT_0(N)^\text{p-old}_{(i)}$ and $\BT_0(N)\twoheadrightarrow\BT_0(N)^\text{p-new}_{(i)}$, which clearly factors as
\begin{align*}
\begin{cases}
\BT_0(N)^\text{p-old}_{(i-1)}\twoheadrightarrow\BT_0(N)^\text{p-old}_{(i)}\\
\BT_0(N)^\text{p-old}_{(i-1)}\twoheadrightarrow\BT_0(N)^\text{p-new}_{(i)}
\end{cases}
\end{align*}
and combine to give an injection
\[\BT_0(N)^\text{p-old}_{(i-1)}\hookrightarrow\BT_0(N)^\text{p-old}_{(i)}\times\BT_0(N)^\text{p-new}_{(i)}.\]
It follows that we obtain an injection
\[\BT_0(N)^\text{p-old}\hookrightarrow\BT_0(N)^\text{p-old}_{(n_p-1)}\times(\prod^{n_p-1}_{i=1}\BT_0(N)^\text{p-new}_{(i)}).\]
Since $J_0(N)^{(n_p)}_\text{p-old}$ is clearly contained in $J_0(N)^{(n_p-1)}_\text{p-old}$, the action of $\BT_0(N)$ factors to give the following $\BZ$-algebra homomorphism
\begin{align*}
  \BT_0(N)^\text{p-old}_{(n_p-1)}\twoheadrightarrow\BT_0(N)^\text{p-old}_{(n_p)}.
\end{align*}
On the other hand, $J_0(N)^{(n_p-1)}_\text{p-old}$ contains the subvariety $\pi^{N*}_{N/p,p}(J_0(N/p))$, which is isogenous to $J_0(N/p)$ as $\ker(\pi^{N*}_{N/p,p})\subseteq\ker([p])$ is finite. Thus it follows from Corollary~\ref{cor2} that there is a surjective $\BZ$-algebra homomorphism
\begin{align*}
  \BT_0(N)^\text{p-old}_{(n_p-1)}\twoheadrightarrow\BT_0(N/p),
\end{align*}
which maps $T^{\Gamma_0(N)}_\ell$ to $T^{\Gamma_0(N/p)}_\ell$ for any prime $\ell$. As $J_0(N)^{(n_p-1)}_\text{p-old}$ is generated by its subvarieties $J_0(N)^{(n_p)}_\text{p-old}$ and $\pi^{N*}_{N/p,p}(J_0(N/p))$, the above two surjections combined to give the following injection
\begin{align*}
  \BT_0(N)^\text{p-old}_{(n_p-1)}\hookrightarrow\BT_0(N)^\text{p-old}_{(n_p)}\times\BT_0(N/p).
\end{align*}
It remains to prove (1). Since $T^{\Gamma_0(N)}_p=\pi^{N*}_{N/p,p}\circ\pi^{N}_{N/p,1*}$ by Lemma~\ref{lemma3}, we find that $T^{\Gamma_0(N)}_p(J_0(N))\subseteq\pi^{N*}_{N/p,p}(J_0(N/p))$, and is therefore contained in $J_0(N)^{(i)}_\text{p-old}$ for any $i$. So the first assertion follows. In particular, we find that $\BT_0(N)^\text{p-new}_{(i)}$ is generated by (the image of) $T^{\Gamma_0(N)}_\ell$ for those prime $\ell\neq p$. However, it follows from the construction that $\pi^{N*}_{N/p^i,1}$ induces a surjective homomorphism
\[J_0(N/p^i)^\text{p-new}\twoheadrightarrow J_0(N)^\text{p-new}_{(i)}.\]
As $\pi^{N}_{N/p^i,1}$ commutes with the $\ell$-th Hecke operators for any prime $\ell\neq p$, we find that there is an induced surjective $\BZ$-algebra homomorphism $\BT_0(N/p^i)^\text{p-new}\twoheadrightarrow\BT_0(N)^\text{p-new}_{(i)}$ which completes the proof.
\end{proof}

\section{Definition and basic properties of maximal Eisenstein ideals}
Let $N$ be a positive integer. If $\fm\subseteq\BT_0(N)$ is a maximal ideal, then there exists a unique semi-simple representation
\[\rho_\fm:~\Gal(\bar{\BQ}/\BQ)\rightarrow\GL_2(\BT_0(N)/\fm),\]
which is unramified outside $N\fm$, such that
\begin{align}\label{4.1}
  \begin{cases}
    \tr(\rho_\fm(\Frob_\ell))=T^{\Gamma_0(N)}_\ell &\pmod\fm\\
    \det(\rho_\fm(\Frob_\ell))=\ell &\pmod\fm
  \end{cases}
\end{align}
for any prime $\ell$ with $(\ell,N\fm)=1$. See \cite{Ri2}.

Let $\fP$ be a minimal prime of $\BT_0(N)$ with $\fP\subseteq\fm$. Then $\CO:=\BT_0(N)/\fP$ is an order in the number field $K:=\CO\otimes_{\BZ}\BQ$. Let $\CO_K$ be the ring of integers in $K$. Note that $\fm$ corresponds to a prime in $\CO$ which will also be denoted as $\fm$. Choose a prime $\lambda$ of $\CO_K$ lying over $\fm$. Let $\kappa(\lambda):=\CO_K/\lambda$ and $\kappa(\fm):=\CO/\fm$. Then we have the following commutative diagram
\[\xymatrix{
                & \CO_K  \ar[r] & \kappa(\lambda)   \\
  \BT_0(N)  \ar[r] &\CO  \ar[u]\ar[r] & \ar[u]\kappa(\fm).            }
\]
Let $f=\sum_{n\geq1}a_n(f)\fq^n\in S^B_2(\Gamma_0(N),\BC)$ be the normalized eigenform corresponding to the ring homomorphism $\BT_0(N)\twoheadrightarrow\CO$. It follows that there is a $\lambda$-adic representation $\rho_{f,\lambda}:\Gal(\bar{\BQ}/\BQ)\rightarrow\GL_2(\CO_{K,\lambda})$ which is unramified outside $N\lambda$, such that for any prime $\ell$ with $(\ell,N\lambda)=1$
\begin{align*}
  \begin{cases}
    \tr(\rho_{f,\lambda}(\Frob_\ell))=a_\ell(f)\\
    \det(\rho_{f,\lambda}(\Frob_\ell))=\ell .
  \end{cases}
\end{align*}
Let $\bar{\rho}^{ss}_{f,\lambda}:\Gal(\bar{\BQ}/\BQ)\rightarrow\GL_2(\kappa(\lambda))$ be the semi-simplification of the reduction $\bar{\rho}_{f\lambda}$ of $\rho_{f,\lambda}$. Then we find by the density theorem and the Brauer-Nesbitt theorem that
\begin{align}\label{4.2}
  \rho_{\fm}\otimes_{\kappa(\fm)}\kappa(\lambda)=\bar{\rho}^{ss}_{f,\lambda}.
\end{align}

\begin{defn}\label{def-Eisenstein}
  Let $N\in\BZ_{\geq1}$ be a positive integer. If $\fm\subseteq\BT_0(N)$ is a maximal ideal such that $\rho_\fm$ is reducible, then we call $\fm$ to be \textbf{a maximal Eisenstein ideal}.
\end{defn}

We fix some notations before going on. Let $\fm\subseteq\BT_0(N)$ be a maximal Eisenstein ideal as above. Denote by $q:=char(\kappa(\fm))$ be the residue characteristic of $\fm$. We always \textbf{assume} that $q\neq2$. Since $\rho_\fm$ is semi-simple and reducible, we have
\[\rho_\fm=\bar{\chi}_1\oplus\bar{\chi}_2,\]
where $\bar{\chi}_i:\Gal(\bar{\BQ}/\BQ)\rightarrow\kappa(\fm)^\times$ is a character for $i=1,2$. Denote by $\epsilon_q$ to be the modulo-$q$ cyclotomic character. Then we find by Eq.~\ref{4.1} that
\[\bar{\chi}_1\cdot\bar{\chi}_2=\epsilon_q.\]
For $i=1,2$ we can write $\bar{\chi}_i=\bar{\eta}_i\cdot\epsilon^{k_i}_q$, where $\bar{\eta_i}$ is unramified at $q$ and $k_i\in\{0,...,q-2\}$. In particular, we find that $K_1+k_2\equiv1\pmod{q-1}$. It follows that $k_1\neq k_2$ as we have assumed $q$ to be odd. We define
\begin{align}\label{4.3}
\bar{\chi}^{-1}:=
  \begin{cases}
    \bar{\chi}_1,& \text{if }k_1<k_2\\
    \bar{\chi}_2,& \text{if }k_2<k_1,
  \end{cases}
\end{align}
which will be called as \text{the character associated with $\fm$}. Thus we have $\rho_\fm=\bar{\chi}^{-1}\oplus\bar{\chi}\cdot\epsilon_q$. Denote by $\ff_{\bar{\chi}}$ to be the conductor of $\bar{\chi}$. Note that by Lemma 4.12 of \cite{DDT}, $\bar{\chi}$ is unramified at $q$ if $(q,N)=1$.

\begin{prop}\label{bound for conductor}
Let $N\in\BZ_{\geq1}$ and $\fm$ be a maximal Eisenstein ideal with residue characteristic $q\neq2$. Let $\bar{\chi}$ be the character associated to $\fm$ and $\ff_{\bar{\chi}}$ be the conductor of $\bar{\chi}$. Then we have $\ff^2_{\bar{\chi}}\mid N$.
\end{prop}
\begin{proof}
Choose $f,\lambda$ as in Eq.~\ref{4.2}, and denote also by $\bar{\chi}$ to be its base change from $\kappa(\fm)$ to $\kappa(\lambda)$. For any prime $p\neq q$, we have
\begin{align*}
\begin{cases}
  v_p(\ff_\chi)=\int^{+\infty}_{-1}\codim(\bar{\chi}^{I_u})du\\
  v_p(\bar{\rho}_{f,\lambda})=\int^{+\infty}_{-1}\codim(\bar{\rho}^{I_u}_{f,\lambda})du\leq v_p(N),
\end{cases}
\end{align*}
where the last inequality follows from Lemma 2.7 and Theorem 3.1,(d) of \cite{DDT}. By symmetry, we may assume that there is an exact sequence $0\rightarrow\bar{\chi}^{-1}\rightarrow\bar{\rho}_{f,\lambda}\rightarrow\bar{\chi}\rightarrow0$ of $I_p$-modules. Then:
\begin{enumerate}
  \item If $\codim(\bar{\rho}^{I_u}_{f,\lambda})=0$, then $\bar{\chi}|_{I_u}$ is trivial so that $2\codim(\bar{\chi}^{I_u})\leq\codim(\bar{\rho}^{I_u}_{f,\lambda})$;
  \item If $\codim(\bar{\rho}^{I_u}_{f,\lambda})=2$, then $2\codim(\bar{\chi}^{I_u})\leq\codim(\bar{\rho}^{I_u}_{f,\lambda})$ because $\codim(\bar{\chi}^{I_u})\leq1$;
  \item Suppose that $\codim(\bar{\rho}^{I_u}_{f,\lambda})=1$. If $\codim(\bar{\chi}^{I_u})\leq1$, or equivalently, $\bar{\chi}|_{I_u}$ is not trivial, then $\bar{\chi}^{-1}\bigcap\bar{\rho}^{I_u}_{f,\lambda}=0$, so the exact sequence induces $\bar{\rho}^{I_u}_{f,\lambda}\cong\bar{\chi}$. But then we find that $\bar{\chi}|_{I_u}=\id$ which is a contradiction. Therefore $\codim(\bar{\chi}^{I_u})=0$ and $2\codim(\bar{\chi}^{I_u})\leq\codim(\bar{\rho}^{I_u}_{f,\lambda})$ also holds.
\end{enumerate}
It follows that $2\cdot v_p(\ff_{\bar{\chi}})\leq v_p(N)$ for any prime $p\neq q$. Thus to complete the proof we only need to show that $2\cdot v_q(\ff_{\bar{\chi}})\leq v_q(n)$. If $q\nmid N$, this follows from Lemma 4.2 of \cite{DDT} which says that $\bar{\chi}$ is unramified.
On the other hand, since $v_q(\ff_{\bar{\chi}})\leq1$, the inequality $v_q(\ff_{\bar{\chi}})\leq v_q(N)$ automatically holds if $q^2\mid N$. So it remains to consider the situation when $q\|N$.
\begin{itemize}
  \item If $q\|N$ and $m$ is $q$-new, then we may assume the modular form $f$ in Eq.~\ref{4.2} to be $q$-new. Then it follows from Theorem 3.1 of \cite{DDT} that $\bar{\rho}_{f,\lambda}$ is ordinary and hence has an unramified quotient. Since $\bar{\rho}^{ss}_{f,\lambda}|_{G_q}=\rho_{\fm}|_{G_q}=\bar{\chi}^{-1}|_{G_q}\oplus\bar{\chi}\cdot\epsilon_q|_{G_q}$, we find that $\bar{\chi}$ is unramified at $q$, whence the inequality;
  \item If $q\|N$ and $m$ is $q$-old, then we have $\BT_0(N)/\fm\simeq\BT_0(N)^{\text{q-old}}/\fm$. Denote by $\fn$ to be the inverse image of $\fm$ in $\BT_0(N/q)$ via $\BT_0(N/q)\rightarrow\BT_0(N)^{\text{q-old}}\simeq\BT_0(N/q)[x]/(x^2-T^{\Gamma_0(N/q)}_q\cdot x-q)$. Then we find by the density theorem and the Brauer-Nesbitt Theorem that $\rho_\fn\simeq\rho_\fm=\bar{\chi}^{-1}\oplus\bar{\chi}\cdot\epsilon_q$. So we also find that $\bar{\chi}$ is unramified at $q$ because $q\nmid(N/q)$, which completes the proof.
\end{itemize}
\end{proof}

The above proposition gives us an upper bound for the conductor $\ff_{\bar{\chi}}$. In particular, we find that $\bar{\chi}$ is unramified outside $qN$. Moreover, it follows from Eq.~(\ref{4.1}) that
\begin{align}
  T^{\Gamma_0(N)}_\ell\equiv\bar{\chi}(\ell)^{-1}+\ell\cdot\bar{\chi}(\ell)\pmod{\fm}
\end{align}
for any prime $\ell\nmid qN$.

\begin{lem}\label{4.4}
Let $N\in\BZ_{\geq1}$ and $p$ be a prime with $p\mid N$ but $p\nmid 2q$. Then:
\begin{enumerate}
  \item $T^{\Gamma_0(N)}_p\equiv\bar{\chi}(p)^{-1}, \ p\cdot\bar{\chi}(p)\pmod{\fm}$ if $p\|N$;
  \item $T^{\Gamma_0(N)}_p\equiv0,\ \bar{\chi}(p)^{-1} \text{ or } p\cdot\bar{\chi}(p)\pmod{\fm}$ if $p^2\mid N$. Moreover we have $T^{\Gamma_0(N)}_p\equiv0\pmod{\fm}$ if $p\mid\ff_{\bar{\chi}}$.
\end{enumerate}
\end{lem}
\begin{proof}
We first prove (1), and we will distinguish into the following situations:

(1.a) If $\fm$ is $p$-new so that we have $\BT_0(N)/\fm\simeq\BT_0(N)^\text{p-new}/\fm$, then we find that $T^{\Gamma_0(N)}_p\equiv\pm1\pmod{\fm}$, and we can choose the modular form $f$ in Eq.~\ref{4.2} to be $p$-new. Thus it follows from Theorem 3.1,(e) if \cite{DDT} that
\begin{align*}
  \rho_{f,p}|_{G_p}\sim\left(
                         \begin{array}{cc}
                           \eta\cdot\epsilon_q & \star \\
                           0 & \eta \\
                         \end{array}
                       \right),
\end{align*}
where $\eta:G_p\rightarrow\{\pm1\}$ is the unique unramified quadratic character such that $\eta(p)=a_p(f)=\pm1$. It follows that $\eta=\bar{\chi}^{-1}|_{G_p}$, and that
\begin{align*}
  T^{\Gamma_0(N)}_p&\equiv\eta(p)\pmod{\fm}\\
  &=\bar{\chi}(p)^{-1}.
\end{align*}

(1.b) If $\fm$ is $p$-old, then we have $\BT_0(N)/\fm\simeq\BT_0(N)^\text{p-old}/\fm$. Recall that $\BT_0(N)^\text{p-old}\simeq\BT_0(N/p)[x]/(x^2-T^{\Gamma_0(N/p)}\cdot x+p)$ as $p$ is odd. Let $\fn$ be the inverse image of $\fm$ under $\BT_0(N/p)\rightarrow\BT_0(N)^\text{p-old}$. Then
\begin{align*}
T^{\Gamma_0(N/p)}_\ell\pmod{\fn}&=T^{\Gamma_0(N)}_\ell\pmod{\fm}\\
&=\bar{\chi}(\ell)^{-1}+\ell\cdot\bar{\chi}(\ell)
\end{align*}
for any prime $\ell\nmid qN$, so it follows from the density theorem and the Brauer-Nesbitt theorem that $\rho_\fn\simeq\rho_\fm=\bar{\chi}^{-1}\oplus\bar{\chi}\cdot\epsilon_q$. In particular, we find that $T^{\Gamma_0(N/p)}_p\equiv\bar{\chi}(p)^{-1}+p\cdot\bar{\chi}(p)\pmod\fn$. It follows that $T^{\Gamma_0(N)}_p\pmod{\fm}$ satisfies the equation $(x-\bar{\chi}(p)^{-1})\cdot(x-p\cdot\bar{\chi}(p))=0$, so that $T^{\Gamma_0(N)}_p\pmod{\fm}=\bar{\chi}(p)^{-1}$ or $p\cdot\bar{\chi}(p)$.

Now we turn to the proof of (2). We will also distinguish into several situations:

(2.a) If $\fm$ is $p$-new, then $\BT_0(N)/\fm\simeq\BT_0(N)^\text{p-new}/\fm$ so that $T^{\Gamma_0(N)}_p\equiv0\pmod{\fm}$.

(2.b) If $\fm$ is $p$-old, then $\BT_0(N)/\fm\simeq\BT_0(N)^\text{p-old}/\fm$. Recall that there is an injection of $\BZ$-algebras
\[\BT_0(N)^\text{p-old}\hookrightarrow(\prod^{n_p-1}_{i=1}\BT_0(N)^\text{p-new}_{(i)})\times\BT_0(N)^\text{p-old}_{(n_p)}\times\BT_0(N/p),\]
where $n_p=v_p(N)\geq2$. Suppose the image of $\fm$ in $\BT_0(N)^\text{p-new}_{(i)}$ is still maximal for some $i=1,2,...,n_p-1$, then $T^{\Gamma_0(N)}_p\equiv0\pmod{\fm}$ because $T^{\Gamma_0(N)}_p\mapsto0$ via $\BT_0(N)^\text{p-old}\rightarrow\BT_0(N)^\text{p-new}_{(i)}$. Similarly, since the image of $T^{\Gamma_0(N)}_p$ in $\BT_0(N)^\text{p-old}_{(n_p)}$ is nilpotent, we find that $T^{\Gamma_0(N)}_p\equiv0\pmod{\fm}$ if $\fm$ stays maximal in $\BT_0(N)^\text{p-old}_{(n_p)}$. Finally, if the image of $\fm$ in $\BT_0(N/p)$ is still maximal, then we have
\[\BT_0(N)/\fm\simeq\BT_0(N)^\text{p-old}/\fm\simeq\BT_0(N/p)/\fm,\]
and the assertion follows inductively from the above results.

(2.c) We still need to show that $T^{\Gamma_0(N)}_p\equiv0\pmod{\fm}$ if $p\mid\ff_{\bar{\chi}}$. By the above inductive process, we may assume that there is an isomorphism $\BT_0(N)/\fm\simeq\BT_0(N^{(p)}\ff^{v_p(N)}_{\bar{\chi}})/\fm$ with $N^{(p)}$ being the prime-to-$p$ part of $N$. Note that by Proposition~\ref{bound for conductor} the map $\BT_0(N)\rightarrow\BT_0(N^{(p)}\ff^{v_p(N)}_{\bar{\chi}})/\fm$ can not factor through $\BT_0(N^{(p)}\ff^{v_p(N)-1}_{\bar{\chi}})$, so we find that $T^{\Gamma_0(N)}_p$ must be congruent to zero modulo $\fm$, which completes the proof of the lemma.
\end{proof}

\section{Eisenstein series associated to a maximal Eisenstein ideal}
\textbf{4.1. }We need need some preliminaries for the construction of the Eisenstein series associated to a maximal Eisenstein ideal. Let $C^{\infty}(\CH,\BC)=\{f:\CH\rightarrow\BC|f\text{ is smooth}\}$ be the space of all $\BC$-valued smooth functions on the $\CH$. Then, for any prime $p$, we let
\[\gamma_p:C^{\infty}(\CH,\BC)\rightarrow C^{\infty}(\CH,\BC),\ g\mapsto g|\left(
                                                                             \begin{array}{cc}
                                                                               p & 0 \\
                                                                               0 & 1 \\
                                                                             \end{array}
                                                                           \right).
\]
If $\chi$ is a Dirichlet character of conductor $\ff_\chi$ and $p$ is a prime with $(p,\ff_{\chi})=1$, then we define
\begin{align*}
\begin{cases}
[p]^+_\chi:=1-\chi(p)\cdot\gamma_p\\
[p]^-_\chi:=1-\frac{1}{p\cdot\chi(p)}\cdot\gamma_p.
\end{cases}
\end{align*}
It is clear that for any two primes $p_1,p_2$ which are prime to $\ff_\chi$, the operators $[p_1]^+_\chi,[p_1]^-_\chi,[p_2]^+_\chi$ and $[p_2]^-_\chi$ commutes with each other. For any integer $M=\prod^{k}_{i=1}p^{n_i}_i$ with $(M,\ff_\chi)=1$, we define
\[[M]^{\pm}_\chi:=[p_1]^{\pm}_\chi\circ...\circ[p_1]^{\pm}_\chi\circ...\circ[p_k]^{\pm}_\chi\circ...\circ[p_k]^{\pm}_\chi,\]
where the composition on the right hand side can in fact be in any order.

\begin{lem}\label{lemma3.3}
Let $\chi$ be a Dirichlet character of conductor $\ff_\chi$, $p\nmid\ff_\chi$ be a prime and $N$ be a positive integer. Then $[p]^{\pm}_\chi M_2(\Gamma_0(N),\BC)\subseteq M_2(\Gamma_0(Np),\BC)$, and we have

(1) $\CT^{\Gamma_0(Np)}_\ell\circ[p]^{\pm}_\chi=[p]^{\pm}_\chi\circ\CT^{\Gamma_0(N)}_\ell$ for any prime $\ell\neq p$;

(2) If $p\nmid N$, then $\CT^{\Gamma_0(Np)}_p\circ[p]^{+}_\chi=\CT^{\Gamma_0(N)}_p-\gamma_p-p\cdot\chi(p)$ and $\CT^{\Gamma_0(Np)}_p\circ[p]^{-}_\chi=\CT^{\Gamma_0(N)}_p-\gamma_p-\chi(p)^{-1}$;

(3) If $p\mid N$, then $\CT^{\Gamma_0(Np)}_p\circ[p]^{+}_\chi=\CT^{\Gamma_0(N)}_p-p\cdot\chi(p)$ and $\CT^{\Gamma_0(Np)}_p\circ[p]^{-}_\chi=\CT^{\Gamma_0(N)}_p-\chi(p)^{-1}$.
\end{lem}
\begin{proof}
Since $\gamma_p$ maps $M_2(\Gamma_0(N),\BC)$ to $M_2(\Gamma_0(Np),\BC)$ and $[p]^{\pm}_\chi$ is defined to be a linear combination of the identity map and $\gamma_p$, we find that $[p]^{\pm}_\chi$ also maps $M_2(\Gamma_0(N),\BC)$ to $M_2(\Gamma_0(Np),\BC)$. Moreover, if $\ell$ is a prime and $\ell\neq p$, then $\gamma_p$ commutes with $\CT_\ell=\sum^{\ell-1}_{k=0}\left(
                                                                                           \begin{array}{cc}
                                                                                             1 & k \\
                                                                                             0 & \ell \\
                                                                                           \end{array}
                                                                                         \right)+\left(
                                                                                                   \begin{array}{cc}
                                                                                                     \ell & 0 \\
                                                                                                     0 & 1 \\
                                                                                                   \end{array}
                                                                                                 \right)
$ (or $\sum^{\ell-1}_{k=0}\left(
         \begin{array}{cc}
           1 & k \\
           0 & \ell \\
         \end{array}
       \right)
$) if $\ell\nmid N$ (or respectively $\ell\mid N$) as operators on corresponding space of modular forms, so the first assertion follows.

If $p\nmid N$, then we have by definition that
\begin{align*}
  \CT^{\Gamma_0(Np)}_p\circ[p]^{+}_\chi(g)&=g|\left[1-\chi(p)\cdot\left(
                                                                    \begin{array}{cc}
                                                                      p & 0 \\
                                                                      0 & 1 \\
                                                                    \end{array}
                                                                  \right)
  \right]|\sum^{p-1}_{k=0}\left(
                               \begin{array}{cc}
                                 1 & k \\
                                 0 & p \\
                               \end{array}
                             \right)\\
  &=g|\sum^{p-1}_{k=0}\left(
                               \begin{array}{cc}
                                 1 & k \\
                                 0 & p \\
                               \end{array}
                             \right)-\chi(p)\cdot g|\sum^{p-1}_{k=0}\left(
                               \begin{array}{cc}
                                 p & pk \\
                                 0 & p \\
                               \end{array}
                             \right)\\
  &=\CT^{\Gamma_0(N)}_p(g)-f|\gamma_p-p\cdot\chi(p)\cdot g,
\end{align*}
for any $g\in M_2(\Gamma_0(N),\BC)$; similarly, we have by definition that
\begin{align*}
  \CT^{\Gamma_0(Np)}_p\circ[p]^{-}_\chi(g)&=g|\left[1-p^{-1}\cdot\chi^{-1}(p)\cdot\left(
                                                                    \begin{array}{cc}
                                                                      p & 0 \\
                                                                      0 & 1 \\
                                                                    \end{array}
                                                                  \right)
  \right]|\sum^{p-1}_{k=0}\left(
                               \begin{array}{cc}
                                 1 & k \\
                                 0 & p \\
                               \end{array}
                             \right)\\
  &=g|\sum^{p-1}_{k=0}\left(
                               \begin{array}{cc}
                                 1 & k \\
                                 0 & p \\
                               \end{array}
                             \right)-p^{-1}\cdot\chi^{-1}(p)\cdot g|\sum^{p-1}_{k=0}\left(
                               \begin{array}{cc}
                                 p & pk \\
                                 0 & p \\
                               \end{array}
                             \right)\\
  &=\CT^{\Gamma_0(N)}_p(g)-f|\gamma_p-\chi^{-1}(p)\cdot g
\end{align*}
so the second assertion follows. The proof of the third assertion is similar and we leave it to the reader.
\end{proof}

\textbf{4.2. }Hereafter we fix an \textbf{odd} positive integer $N$ and a maximal Eisenstein ideal $\fm\subset\BT_0(N)$ of residue characteristic $q\nmid6N$. Let $\bar{\chi}$ be the associated character which is of conductor $\ff_{\bar{\chi}}$. We denote by $\chi$ to be the Teichm$\ddot{u}$ller lifting of $\bar{\chi}$ which is of conductor $\ff_{\chi}=\ff_{\bar{\chi}}$. Since $\ff^2_{\bar{\chi}}\mid N$ by Lemma~\ref{bound for conductor}, we can decompose $N$ as
\begin{align}\label{eq4.5}
  N=N_{\bar{\chi}}\cdot(D\cdot C\cdot C_1\cdot\cdot\cdot C_r),
\end{align}
where $N_{\bar{\chi}}:=\prod_{p\mid\ff_{\bar{\chi}}}p^{v_p(N)}$ and $1\leq C_r|\cdot\cdot\cdot|C_1|C|D$ are all square-free positive integers. Let
\begin{align*}
\begin{cases}
  \CP_1(\fm):=\{p\mid D:T^{\Gamma_0(N)}_p\equiv0,\bar{\chi}(p)^{-1}\}\\
  \CP_2(\fm):=\{p\mid D:T^{\Gamma_0(N)}_p\equiv0,p\bar{\chi}(p)\neq\bar{\chi}(p)^{-1}\}\\
  \bar{M}:=\prod_{p\in\CP_1(\fm)}p,\ \bar{L}:=\prod_{p\in\CP_2(\fm)}p,
\end{cases}
\end{align*}
then we define
\begin{align}\label{eq4.6}
  M:=\ff_{\bar{\chi}}\cdot\bar{M},\ L:=\ff_{\bar{\chi}}\cdot\bar{L}.
\end{align}

\begin{lem}\label{4.5}
Let $N$ be an odd positive integer and $\fm\subset\BT_0(N)$ a maximal ideal with $(\fm,6N)=1$. Let notations be as in Eqs.~(\ref{eq4.5})-(\ref{eq4.6}). Then
\begin{enumerate}
  \item $M>1$;
  \item $T^{\Gamma_0(N)}_q\equiv\bar{\chi}(q)^{-1}\pmod{\fm}$.
\end{enumerate}
\end{lem}
\begin{proof}
To prove (1) it is enough to show that $M>1$. Suppose to the contrary that $M=1$, so we have $\bar{\chi}=1$ and $T^{\Gamma_0(N)}_p\equiv p\neq1\pmod{\fm}$ for any prime $p\mid N$.

Let $p$ be a prime divisor of $N$ such that $n_p=v_p(N)\geq2$. Then because $T^{\Gamma_0(N)}_p$ is mapped to $0$ via $\BT_0(N)\rightarrow\BT_0(N)^\text{p-new}$, we find that the image of $\fm$ in $\BT_0(N)^\text{p-new}$ contains $T^{\Gamma_0(N)}_p-p=-p$ which implies that $\BT_0(N)^\text{p-new}_\fm=0$. Ir follows that $\fm$ must be $p$-old. Similarly, we find that $(\BT_0(N)^\text{p-new}_{(i)})_\fm=0$ for any $i=1,...,n_p-1$ and also that $(\BT_0(N)^\text{p-old}_{n_p})_\fm=0$. Therefore the image of $\fm$ in $\BT_0(N/p)$ must be a maximal ideal. Thus it follows inductively that there exists a maximal ideal $\bar{\fm}$ in the Hecke algebra $\BT_0(\bar{N})$, where $\bar{N}=S(N)$ is square-free, such that $T^{\Gamma_0(\bar{N})}\equiv1+\ell\pmod{\bar{\fm}}$ for any prime $\ell\nmid\bar{N}$, $T^{\Gamma_0(\bar{N})}\equiv p\pmod{\bar{\fm}}$ for any $p\mid\bar{N}$.

If $\bar{\fm}$ is $p$-old for some prime $p\mid\bar{N}$, then we have that
\begin{align*}
  \BT_0(\bar{N})/\bar{\fm}&\simeq\BT_0(\bar{N})^\text{p-old}/\bar{\fm}\\
  &\simeq\BT_0(\bar{N}/p)[x]/(x^2-T^{\Gamma_0(\bar{N}/p)}_p\cdot x+p,\ \bar{\fm}),
\end{align*}
which implies that there is a maximal Eisenstein ideal $\bar{\fm}_1\subseteq\BT_0(\bar{N}/p)$ such that $T^{\Gamma_0(N)}_{p'}\equiv p'\neq1\pmod{\bar{\fm}_1}$ for any prime $p'\mid(\bar{N}/p)$. Proceeding in this way, we will arrive at some divisor $d\mid\bar{N}$, such that there exists a maximal Eisenstein ideal $\bar{\fn}\subset\BT_0(d)$ which is $p$-new for any $p\mid d$ and satisfies $T^{\Gamma_0(d)}\equiv p\neq1\pmod{\fn}$ for any $p\mid d$. But this is impossible by Theorem 2.6,(ii) of \cite{BD}. We have thus proved (1).

To prove (2), we need to show that $a_q(f)\equiv\chi(q)^{-1}\pmod{\lambda}$ with $f,\lambda$ as in Eq.~\ref{4.2}. By (1), $E_{M,L,\chi}$ is an Eisenstein series in $\CE_2(\Gamma_0(N),\BC)$. Then it follows from Proposition~\ref{Hecke eigen} that $E_{M,L,\chi}-f$, when modulo $\lambda$, gives rise to a form $\sum^{\infty}_{n=0}a_{qn}\cdot \fq^{qn}\in M^B_2(\Gamma_0(N),\bar{\BF}_q)$. By the main result of \cite{K} we find that this form must be zero, so that $a_q(f)\equiv a_q(E_{M,L,\chi})\equiv\chi(q)^{-1}\pmod{\lambda}$ which completes the proof.
\end{proof}

\begin{defn}\label{key definition}
Notation are as above. Define
\[E_{M,L,\chi}:=[\bar{L}]^-_\chi\circ[\bar{M}]^+_\chi(E_\chi),\]
which is called as \textbf{the Eisenstein series associated with $\fm$}. Here
\begin{align*}
  E_\chi:=-\frac{1}{2g(\chi)}\sum_{a\in(\BZ/\ff_\chi\BZ)^\times}\sum_{b\in(\BZ/\ff^2_\chi\BZ)^\times}\chi(a)\cdot\chi(b)\cdot\phi_{(\frac{a}{\ff_\chi},\frac{b}{\ff^2_\chi})},
\end{align*}
with $g(\chi):=\sum_{t\in\BZ/\ff_\chi\BZ}\chi(t)\cdot e^{2\pi it}$ being the Gauss sum of $\chi$.
\end{defn}

Let us have a closer look at the above functions. From Eq.(\ref{3.2}), we find that
\begin{align*}
E_\chi=&-\frac{\delta_\chi}{4\pi i(z-\overline{z})}-\frac{1}{4g(\chi)}\sum_{a\in({\BZ}/{\ff_\chi\BZ})^\times}\sum_{b\in({\BZ}/{\ff_\chi^2\BZ})^\times}\chi(a)\cdot\chi(b)\cdot B_2(\frac{a}{\ff_\chi})\\
&+\frac{1}{g(\chi)}\sum_{a\in({\BZ}/{\ff_\chi\BZ})^\times}\sum_{b\in({\BZ}/{\ff_\chi^2\BZ})^\times}\chi(a)\cdot\chi(b)\cdot P_{(\frac{a}{\ff_\chi},\frac{b}{\ff_\chi^2})},
\end{align*}
where $\delta_\chi=1$ or $0$ according to $\chi=1$ or not. By Eq.~(\ref{3.3}), we have
\begin{align*}
\sum_{a\in({\BZ}/{\ff_\chi\BZ})^\times}\sum_{b\in({\BZ}/{\ff_\chi^2\BZ})^\times}\chi(a)\cdot\chi(b)\cdot P_{(\frac{a}{\ff_\chi},\frac{b}{\ff_\chi^2})}&=\sum^{\infty}_{k,m=1}\frac{k\chi(k)}{\ff_\chi}\left(\sum_{y\in({\BZ}/{\ff_\chi^2\BZ})^{\times}}\chi(y)e^{2\pi i\frac{my}{\ff_\chi^2}}\right)e^{2\pi i\frac{mk}{\ff_\chi}z}\\
&=\sum^{\infty}_{k,m=1}\frac{k\chi(k)}{\ff_\chi}\left(\sum_{y\in({\BZ}/{\ff_\chi^2\BZ})^{\times}}\chi(y)e^{2\pi i\frac{my}{\ff_\chi}}\right)e^{2\pi imkz}\\
&=g(\chi)\sum^{\infty}_{k,m=1}k\cdot\chi(k)\cdot\chi^{-1}(m)\cdot e^{2\pi imkz},
\end{align*}
where $\chi(n)$ is defined to be $0$ when $(n,\ff_\chi)\neq1$ as usual, it follows that
\begin{align}\label{3.5}
  E_\chi=-\frac{\delta_\chi}{4\pi i(z-\overline{z})}+a_0(E_\chi;[\infty])+\sum^{\infty}_{n=1}\sigma_{\chi}(n)\cdot \fq^n,
\end{align}
with
\begin{align}\label{3.6}
a_0(E_\chi;[\infty])=
\begin{cases}
-\frac{1}{24} &,\text{if}\ \chi=1\\
\ \ 0 &,\text{otherwise}\
\end{cases},
\end{align}
and
\begin{align}\label{3.7}
\sigma_\chi(n):=
\sum_{1\leq d\mid n}d\cdot\chi(d)\cdot\chi^{-1}({n}/{d}).
\end{align}
In particular, we find that $a_1(E_\chi;[\infty])=1$, that is to say, $E_\chi$ is normalized. Since the operators $[\bar{M}]^+_\chi$ and $[\bar{L}]^-_{\chi}$ do not change the first term of the $\fq$-expansion, we find that $E_{M,L,\chi}$ is also normalized. Moreover, we claim that $E_{M,L,\chi}$ is holomorphic and hence is an Eisenstein series of level $ML$. This is clear if $\chi\neq1$. On the other hand, if $\chi=1$, then we have $\ff_\chi=1$. Since $M>1$ and $[M]^+(\frac{1}{z-\overline{z}})=0$ , we find ny Eq.~(\ref{3.5}) that $E_{M,L,1}$ is also holomorphic which completes the proof of the claim.

\begin{lem}\label{3.8}
For any Dirichlet character $\chi$ of conductor $\ff_\chi$, we have
\begin{align*}
\CT^{\Gamma_0(\ff^2_\chi)}_\ell(E_\chi)=
\begin{cases}
  [\chi(\ell)^{-1}+\ell\cdot\chi(\ell)]\cdot E_\chi &, \text{if}\ \ell\nmid\ff_\chi\\
  0 &, \text{if}\ \ell\mid\ff_\chi.
\end{cases}
\end{align*}
\end{lem}
\begin{proof}
If $\ell\nmid\ff_\chi$ is a prime, let $\ell'\in\BZ$ such that $\ell\cdot\ell'\equiv1\pmod{\ff_\chi}$, then by \cite{St} Proposition 2.4.7 we have that
\[\CT^{\Gamma_0(\ff^2_\chi)}_\ell(\phi_{(\frac{a}{\ff_\chi},\frac{b}{\ff^2_\chi})})=\phi_{(\frac{a}{\ff_\chi},\frac{b\ell}{\ff^2_\chi})}+\ell\cdot\phi_{(\frac{a\ell'}{\ff_\chi},\frac{b}{\ff^2_\chi})}.\]
It follows that
\begin{align*}
\CT^{\Gamma_0(\ff^2_\chi)}_\ell(E_\chi)&=\sum_{a\in\BZ/\ff_\chi\BZ}\sum_{b\in\BZ/\ff^2_\chi\BZ}\chi(a)\chi(b)\left(\phi_{(\frac{a}{\ff_\chi},\frac{b\ell}{\ff^2_\chi})}+\ell\cdot\phi_{(\frac{a\ell'}{\ff_\chi},\frac{b}{\ff^2_\chi})}\right)\\
&=\left(\chi^{-1}(\ell)+\ell\cdot\chi(\ell)\right)\cdot E_\chi,
\end{align*}
which proves the assertion for those primes not dividing $\ff_\chi$.

On the other hand, by the distribution law, we have that
\begin{align*}
  E_\chi&=\sum_{a\in\BZ/\ff_\chi\BZ}\sum_{b\in\BZ/\ff^2_\chi\BZ}\chi(a)\chi(b)\phi_{(\frac{a}{\ff_\chi},\frac{b}{\ff^2_\chi})}\\
  &=\sum_{a,b\in\BZ/\ff_\chi\BZ}\chi(a)\chi(b)\phi_{(\frac{a}{\ff_\chi},\frac{b}{\ff_\chi})}|\left(
                                                                                               \begin{array}{cc}
                                                                                                 \ff_\chi & 0 \\
                                                                                                 0 & 1 \\
                                                                                               \end{array}
                                                                                             \right).
\end{align*}
Thus, if $\ell$ is a prime with $\ell\mid\ff_\chi$, then
\begin{align*}
  \CT^{\Gamma_0(\ff_\chi)}(E_\chi)&=\sum_{a,b\in\BZ/\ff_\chi\BZ}\chi(a)\chi(b)\phi_{(\frac{a}{\ff_\chi},\frac{b}{\ff_\chi})}|\left(
                                                                                               \begin{array}{cc}
                                                                                                 \ff_\chi & 0 \\
                                                                                                 0 & 1 \\
                                                                                               \end{array}
                                                                                             \right)\sum^{\ell-1}_{k=0}\left(
                                                                                                                         \begin{array}{cc}
                                                                                                                           1 & k \\
                                                                                                                           0 & \ell \\
                                                                                                                         \end{array}
                                                                                                                       \right)\\
  &=\left[\sum_{a,b\in\BZ/\ff_\chi\BZ}\chi(a)\chi(b)\phi_{(\frac{a}{\ff_\chi},\frac{b}{\ff_\chi})}|\sum^{\ell-1}_{k=0}\left(
                                                                                               \begin{array}{cc}
                                                                                                 1 & \frac{\ff_\chi}{\ell}k \\
                                                                                                 0 & 1 \\
                                                                                               \end{array}
                                                                                             \right)\right]|\left(
                                                                                                                         \begin{array}{cc}
                                                                                                                           \frac{\ff_\chi}{\ell} & 0 \\
                                                                                                                           0 & 1 \\
                                                                                                                         \end{array}
                                                                                                                       \right)\\
  &=\left[\sum_{a,b\in\BZ/\ff_\chi\BZ}\chi(a)\chi(b)\phi_{(\frac{a}{\ff_\chi},\frac{b}{\ff_\chi}+\sum^{\ell-1}_{k=0}\frac{ak}{\ell})}\right]|\left(
                                                                                                                         \begin{array}{cc}
                                                                                                                           \frac{\ff_\chi}{\ell} & 0 \\
                                                                                                                           0 & 1 \\
                                                                                                                         \end{array}
                                                                                                                       \right),
\end{align*}
with the function $\BZ/\ff_\chi\rightarrow\BC,\ b\mapsto\sum_{a\in\BZ/\ff_\chi\BZ}\chi(a)\phi_{(\frac{a}{\ff_\chi},\frac{b}{\ff_\chi}+\sum^{\ell-1}_{k=0}\frac{ak}{\ell})}$, depends only on $b\pmod{\ff_\chi/\ell}$. So we find that the above sum is zero as $\chi$ is primitive of conductor $\ff_\chi$ and hence completes the proof.
\end{proof}

\begin{prop}\label{Hecke eigen}
Let $N$ be an odd positive integer and $\fm\subseteq\BT_0(N)$ be a maximal Eisenstein ideal with $(\fm,6N)=1$. Let $E_{M,L,\chi}$ be the Eisenstein series associated with $\fm$ as in Definition~\ref{key definition}. Then we have
\begin{align*}
  \CT^{\Gamma_0(N)}_\ell(E_{M,L,\chi})=
\begin{cases}
      \left(\chi^{-1}(\ell)+\ell\cdot\chi(\ell)\right)\cdot E_{M,L,\chi} &, \text{if } \ell\nmid N\\
      \chi^{-1}(\ell)\cdot E_{M,L,\chi} &, \text{if } \ell\mid M/(M,L)\\
      \ell\cdot\chi(\ell)\cdot E_{M,L,\chi}&, \text{if } \ell\mid L/(M,L)\\
      0 &, \text{if } \ell\mid (M,L);
    \end{cases}
\end{align*}
\end{prop}
\begin{proof}
The assertion follows from Lemma~\ref{lemma3.3},(1) and Lemma~\ref{3.8},(1). If $p$ is a prime with $p\mid N$, then both $\CT^{\Gamma_0(N)}_p$ and $\CT^{\Gamma_0(ML)}_p$ are analytically given as $\sum^{p-1}_{k=0}\left(
                   \begin{array}{cc}
                     1 & k \\
                     0 & p \\
                   \end{array}
                 \right)
$, so we have that
\[\CT^{\Gamma_0(N)}_p(E_{M,L,\chi})=\CT^{\Gamma_0(ML)}_p(E_{M,L,\chi})\]
Thus we find that:

$\bullet$ If $p\mid M/(M,L)$, or equivalently, $p\mid \bar{M}/(\bar{M},\bar{L})$, then we have by (1) and (2) of Lemma~\ref{lemma3.3} that
\begin{align*}
&\CT^{\Gamma_0(ML)}_p\circ[\bar{L}]^-_\chi\circ[\bar{M}]^+_\chi(E_\chi)\\
=&[\bar{L}]^-_\chi\circ[\bar{M}/p]^+_\chi\circ(\CT^{\Gamma_0(\ff^2_\chi)}_p-\gamma_p-p\cdot\chi(p))(E_\chi)\\
=&[\bar{L}]^-_\chi\circ[\bar{M}/p]^+_\chi\circ(\chi^{-1}(p)-\gamma_p)(E_\chi)\\
=&\chi^{-1}(p)\cdot[\bar{L}]^-_\chi\circ[\bar{M}/p]^+_\chi\circ[p]^+_\chi(E_\chi)\\
=&\chi^{-1}(p)\cdot[\bar{L}]^-_\chi\circ[\bar{M}]^+_\chi(E_\chi),
\end{align*}
and therefore $\CT^{\Gamma_0(N)}_p(E_{M,L,\chi})=\chi^{-1}(p)\cdot E_{M,L,\chi}$.

$\bullet$ If $p\mid L/(M,L)$, or equivalently, $p\mid \bar{L}/(\bar{M},\bar{L})$, then we have by (1) and (2) of Lemma~\ref{lemma3.3} that
\begin{align*}
&\CT^{\Gamma_0(ML)}_p\circ[\bar{L}]^-_\chi\circ[\bar{M}]^+_\chi(E_\chi)\\
=&[\bar{L}/p]^-_\chi\circ[\bar{M}]^+_\chi\circ(\CT^{\Gamma_0(\ff^2_\chi)}_p-\gamma_p-\chi^{-1}(p))(E_\chi)\\
=&[\bar{L}/p]^-_\chi\circ[\bar{M}]^+_\chi\circ(p\cdot\chi(p)-\gamma_p)(E_\chi)\\
=&p\cdot\chi(p)\cdot[\bar{L}/p]^-_\chi\circ[\bar{M}]^+_\chi\circ[p]^-_\chi(E_\chi)\\
=&p\cdot\chi(p)\cdot[\bar{L}]^-_\chi\circ[\bar{M}]^+_\chi(E_\chi),
\end{align*}
and therefore $\CT^{\Gamma_0(N)}_p(E_{M,L,\chi})=p\cdot\chi(p)\cdot E_{M,L,\chi}$.

$\bullet$ Now if $p\mid(M,L)$ but $p\nmid\ff_\chi$, then we have by (3) of Lemma~\ref{lemma3.3} that
\begin{align*}
&\CT^{\Gamma_0(ML)}_p\circ[\bar{L}]^-_\chi\circ[\bar{M}]^+_\chi(E_\chi)\\
=&(\CT^{\Gamma_0(ML/p)}_p-\chi^{-1}(p))\circ[\bar{L}/p]^-_\chi\circ[\bar{M}]^+_\chi(E_\chi)\\
=&(\chi^{-1}(p)-\chi^{-1}(p))\circ[\bar{L}/p]^-_\chi\circ[\bar{M}]^+_\chi(E_\chi)=0.
\end{align*}

$\bullet$ Finally if $p\mid\ff_\chi$, then $(p,\bar{M})=(p,\bar{L})=1$. So we find by Lemma~\ref{3.3} and Lemma~\ref{3.8} that
\begin{align*}
&\CT^{\Gamma_0(ML)}_p\circ[\bar{L}]^-_\chi\circ[\bar{M}]^+_\chi(E_\chi)\\
=&[\bar{M}]^-_\chi\circ[\bar{L}]^+_\chi\circ\CT^{\Gamma_0(\ff^2_\chi)}_p(E_\chi)=0.
\end{align*}
We have thus completed the proof of the proposition.
\end{proof}

\textbf{3.4. }Fix a positive integer $N$. For any $1\leq d\mid N$, let $\fB_d$ be a full set of representatives for $(\BZ/(d,N/d)\cdot\BZ)^\times$. Then we have the following

\begin{lem}\label{representative}
  $cusp(\Gamma_0(N))=\{[\frac{dx}{N}]|1\leq d\mid N,x\in{\fB_d}\}$.
\end{lem}
\begin{proof}
Since $\#cusp(\Gamma_0(N))=\sum_{1\leq d\mid N}\varphi(d,N/d)$, it is enough to prove that $[\frac{d_1x_1}{N}]=[\frac{d_2x_2}{N}]$ implies that $d_1=d_2$ and $x_1=x_2$.

So suppose that $[\frac{d_1x_1}{N}]=[\frac{d_2x_2}{N}]$, then there exists some $\gamma\in\Gamma_0(N)$ such that $\frac{d_2x_2}{N}=\gamma(\frac{d_1x_1}{N})$. It follows that
\[\left(
    \begin{array}{cc}
      x_2 & u_2 \\
      N/d_2 & v_2 \\
    \end{array}
  \right)(\infty)=\gamma\cdot\left(
                               \begin{array}{cc}
                                 x_1 & u_1 \\
                                 N/d_1 & v_1 \\
                               \end{array}
                             \right)(\infty),
\]
where $\left(
         \begin{array}{cc}
           x_i & u_i \\
           N/d_i & v_i \\
         \end{array}
       \right)\in\SL_2(\BZ)
$ for each $i=1,2$, so that
\[\left(
    \begin{array}{cc}
      x_2 & u_2 \\
      N/d_2 & v_2 \\
    \end{array}
  \right)\cdot\left(
                \begin{array}{cc}
                  1 & n \\
                  0 & 1 \\
                \end{array}
              \right)
  =\pm\gamma\cdot\left(
                               \begin{array}{cc}
                                 x_1 & u_1 \\
                                 N/d_1 & v_1 \\
                               \end{array}
                             \right)(\infty)
\]
for some $n\in\BZ$. Denote $\pm\gamma=\left(
                                         \begin{array}{cc}
                                           a & b \\
                                           c & e \\
                                         \end{array}
                                       \right)
.$ Then we find that $N/d_2=cx_1+e(N/d_1)$, which implies that $(N/d_1)\mid(N/d_2)$ and vice versa, and therefore $d_1=d_2\triangleq d$. It then follows that
\[x_2=ax_1+b(N/d).\]
Since $N/d=cx_1+e(N/d)$, we find that $e\equiv1\pmod{d}$ so that $a\equiv1\pmod{d}$. Thus we find that
\[x_2\equiv x_1\pmod{(d,N/d)},\]
which implies that $x_1=x_2\in\fB_d$ and completes the proof of the lemma.
\end{proof}

\begin{remark}
In the following of this paper, we will always use the above representatives for cusps of $X_0(N)$. Note that, if it is necessary, we may even assume that $x$ to be prime to $N$.
\end{remark}

For any Dirichlet character $\chi$ of conductor $\ff_\chi$, let $E_\chi$ be as in Definition~\ref{key definition}.
These formulas can be used to determine the constant terms of $E_\chi$. Extend $\chi$ to be a function on $\BZ$ so that $\chi(n)=0$ if $(n,\ff_\chi)\neq1$. For any cusp $[\frac{dx}{\ff^2_\chi}]\in X_0(\ff^2_\chi)$ as in Lemma~\ref{representative}, choose $\left(
                                                                                                                     \begin{array}{cc}
                                                                                                                       x & u \\
                                                                                                                       {\ff^2_\chi}/{d} & v \\
                                                                                                                     \end{array}
                                                                                                                   \right)
\in\SL_2(\BZ)$ which maps $[\infty]$ to $[\frac{s^2tx}{\ff^2_\chi}]$. Then
\begin{align*}
  a_0(E_\chi;[\frac{s^2tx}{\ff^2_\chi}])&=-\frac{1}{4g(\chi)}\sum_{a\in({\BZ}/{\ff_\chi\BZ})^\times}\sum_{b\in({\BZ}/{\ff^2_\chi\BZ})^\times}\chi(a)\cdot\chi(b)\cdot B_2(\frac{xa}{\ff_\chi}+\frac{b}{d})\\
  &=-\frac{1}{4g(\chi)}\sum_{b\in({\BZ}/{\ff^2_\chi\BZ})^\times}\chi(b)\left(\sum_{a\in({\BZ}/{\ff_\chi\BZ})^\times}\chi(a)\cdot B_2(\frac{xa}{\ff_\chi}+\frac{b}{d})\right).
\end{align*}
Since the function in the above bracket depends only on $b$ modulo $d$ and $\chi$ is primitive of conductor $\ff_\chi$, we find that $a_0(E_\chi;[\frac{s^2tx}{\ff^2_\chi}])$ must be zero unless $\ff_\chi\mid d$. Moreover, if $d=\ff_\chi\cdot h$ is divided by $\ff_\chi$, then
\begin{align*}
  a_0(E_\chi;[\frac{dx}{\ff^2_\chi}])&=a_0(E_\chi;[\frac{hx}{\ff_\chi}])\\
  &=-\frac{1}{4g(\chi)}\sum_{a\in({\BZ}/{\ff_\chi\BZ})^\times}\sum_{b\in({\BZ}/{\ff^2_\chi\BZ})^\times}\chi(a)\cdot\chi(b)\cdot B_2(\frac{xa}{\ff_\chi}+\frac{b}{h\ff_\chi})\\
  &=-\frac{1}{4g(\chi)}\sum_{a\in({\BZ}/{\ff_\chi\BZ})^\times}\chi(a)\left(\sum_{b\in({\BZ}/{\ff^2_\chi\BZ})^\times}\chi(b)\cdot B_2(\frac{xa}{\ff_\chi}+\frac{b}{h\ff_\chi})\right),
\end{align*}
with the function in the bracket depends only on $a$ modulo $\ff_\chi/h$, so we find that the constant term is zero unless $h=1$. It follows that
\begin{align}
a_0(E_\chi;[\frac{dx}{\ff^2_\chi}])=
\begin{cases}
\chi^{-1}(x)\cdot n_\chi&,\text{if}\ d=\ff_\chi\\
0 &,\text{otherwise,}\
\end{cases}
\end{align}
where
\[n_\chi:=-\frac{\ff_\chi}{4g(\chi)}\sum_{a,b\in{\BZ}/{\ff_\chi\BZ}}\chi(a)\cdot\chi(b)\cdot B_2(\frac{a+b}{\ff_\chi}).\]

\section{Proof of the main theorem}
Recall that, to a maximal Eisenstein ideal $\fm$ with residue characteristic $q\nmid6N$, we can associated an Eisenstein series $E_{M,L,\chi}$, where $M,L$ are as in Eq.~(\ref{4.5}) and $\chi$ is the Teichmuller lifting of $\bar{\chi}$. And we see from Proposition~\ref{Hecke eigen} that $E_{M,L,\chi}$ is an eigenform. Let $\CT_0(N)[\chi]:=\CT_0(N)\otimes_{\BZ}\BZ[\chi]$ and $\BT_0(N)[\chi]:=\BT_0(N)\otimes_{\BZ}\BZ[\chi]$. Then we define
\begin{align}
  I^{(N)}_{M,L,\chi}:=\Im\left(\Ann_{\CT_0(N)[\chi]}(E_{M,L,\chi})\rightarrow\BT_0(N)[\chi]\right).
\end{align}

We are now going to study the quotient $\BT_0(N)[\chi]/I^{(N)}_{M,L,\chi}$ by relating it with $C_{\Gamma_0(N)}(E_{M,L,\chi})$ via the method of Stevens (see \cite{St2}). In general, let $g=\sum^{\infty}_{n=0}a_n\cdot\fq^{n}$ be a weight-$2$ modular form of some level. Then for any Dirichlet character $\eta$ and any prime $p$ with $(p,\ff_{\eta}\cdot\ff_{\chi})=1$, we have
\begin{align*}
[p]^+_{\chi}(g)&=g-\chi(p)\cdot g|\gamma_p\\
&=\sum^{\infty}_{n=0}(a_n-\chi(p)\cdot p\cdot a_{n/p})\cdot\fq^n,
\end{align*}
and therefore
\begin{align*}
  L([p]^+_{\chi}(g),\eta)&=\sum^{\infty}_{n=0}(a_n-\chi(p)\cdot p\cdot a_{n/p})\cdot\eta(n)\cdot n^{-s}\\
  &=(1-\chi(p)\cdot\eta(p)\cdot p^{1-s})\cdot L(g,\eta,s).
\end{align*}
Similarly, since we have
\begin{align*}
[p]^-_{\chi}(g)&=g-\frac{1}{p\cdot\chi(p)}\cdot g|\gamma_p\\
&=\sum^{\infty}_{n=0}(a_n-\chi(p)^{-1}\cdot a_{n/p})\cdot\fq^n,
\end{align*}
it follows that
\begin{align*}
  L([p]^-_{\chi}(g),\eta)&=\sum^{\infty}_{n=0}(a_n-\chi(p)^{-1}\cdot a_{n/p})\cdot\eta(n)\cdot n^{-s}\\
  &=(1-\chi(p)^{-1}\cdot\eta(p)\cdot p^{-s})\cdot L(g,\eta,s).
\end{align*}
Thus we find by Eq.~(\ref{3.5}) that, for any Dirichlet character $\eta$ with conductor $\ff_\eta$ prime to $N$, we have
\begin{align*}
  L(E_{M,L,\chi})=&\prod_{p\mid(L/\ff_\chi)}(1-\frac{\chi(p)^{-1}\eta(p)}{p^s})\cdot\prod_{p\mid(M/\ff_\chi)}(1-\frac{\chi(p)\eta(p)}{p^{s-1}})\\
  &\cdot L(\chi^{-1}\eta,s)\cdot L(\chi\eta,s-1).
\end{align*}
Let $S$ be the set of all primes in the arithmetic progression $-1+4N\BZ$. Then for any $\eta\in\fX^{\infty}_\S$ we have that
\begin{align*}
  \Lambda(E_{M,L,\chi},\eta,1)=&\prod_{p\mid(L/\ff_\chi)}(1-{\chi(p)^{-1}\eta(p)}{p^{-1}})\cdot\prod_{p\mid(M/\ff_\chi)}(1-{\chi(p)\eta(p)})\\
  &\cdot ({g(\bar{\eta})}/{2\pi i})\cdot L(\chi^{-1}\eta,1)\cdot L(\chi\eta,0).
\end{align*}
Then straightforward calculation yields that
\begin{align*}
   \Lambda_{\pm}(E_{M,L,\chi},\eta,1)=&\pm\tilde{\eta}(-1)\cdot\chi(\ff_{\tilde{\eta}})^{-1}\cdot\tilde{\eta}(\ff_\chi)\cdot({g(\chi^{-1})}/{L})\cdot (\frac{1}{2}B_{1,\overline{\chi^{-1}\tilde{\eta}}})\cdot(\frac{1}{2}B_{1\chi\tilde{\eta}})\\
   &\cdot\prod_{p\mid(L/\ff_\chi)}(p-{\chi(p)^{-1}\tilde{\eta}(p)})\cdot\prod_{p\mid(M/\ff_\chi)}(1-{\chi(p)\tilde{\eta}(p)}),
\end{align*}
where $\tilde{\eta}=\eta$ or $\eta\cdot(\frac{\cdot}{p_\eta})$ according as $\chi\eta(-1)=-1$ or not.

\begin{lem}\label{4.6}
Notations are as above. Then we have
\[\CP_{\Gamma_1(N)}(E_{M,L,\chi})=\frac{g(\chi^{-1})}{L}\cdot\BZ[\chi]+\CR_{\Gamma_1(N)}(E_{M,L,\chi}).\]
\end{lem}
\begin{proof}
  It is clear from Theorem 1.3,(b) and Theorem 4.2,(b) of \cite{St2}, together with the above discussion, that $\CP_{\Gamma_1(N)}(E_{M,L,\chi})$ is contained in $\frac{g(\chi^{-1})}{L}\cdot\BZ[\chi]+\CR_{\Gamma_1(N)}(E_{M,L,\chi})$ as sub-$\BZ[\chi]$-modules in $\BQ[\chi]$.

  On the other hand, for any prime $\fP$ in $\BZ[\chi]$, we may choose $\eta\in\fX^{\infty}_{S}$ such that $\prod_{p\mid(L/\ff_\chi)}(p-{\chi(p)^{-1}\tilde{\eta}(p)}),\prod_{p\mid(M/\ff_\chi)}(1-{\chi(p)\tilde{\eta}(p)}),
  \frac{1}{2}B_{1,\overline{\chi^{-1}\tilde{\eta}}}$ and $\frac{1}{2}B_{1\chi\tilde{\eta}}$ are all $\fP$-units by Theorem 4.2,(a) and (c) of \cite{St2}. Then we find by Theorem 1.3,(b) of \cite{St2} again that $\CP_{\Gamma_1(N)}(E_{M,L,\chi})=\frac{g(\chi^{-1})}{L}\cdot\BZ[\chi]+\CR_{\Gamma_1(N)}(E_{M,L,\chi})$.
\end{proof}

\begin{cor}\label{4.7}
  Notations are as above, then we have
  \[C_{\Gamma_0(N)}(E_{M,L,\chi})\otimes_{\BZ}\BZ[\frac{1}{6N}]\simeq\BZ[\frac{1}{6N},\chi]/\fa_{M,L,\chi},\]
  where $\fa_{M,L,\chi}\subseteq\BZ[1/6N,\chi]$ is the ideal generated by the constant terms of $E_{M,L\chi}$.
\end{cor}
\begin{proof}
If $\chi=1$ so that $C_{\Gamma_0(N)}(E_{M,L,\chi})$ is $\BQ$-rational, then $C_{\Gamma_0(N)}(E_{M,L,\chi})\bigcap\sum_N$ is both $\BQ$-rational and of multiplicative type, and is hence contained in $\mu_2$. On the other hand, if $\chi\neq1$, then $C_{\Gamma_0(N)}(E_{M,L,\chi})$ is annihilated by $T^{\Gamma_0(N)}_p$ for any $p\mid\ff_\chi$; since $T^{\Gamma_0(N)}_p$ acts as multiplication by $p$ on $\sum_N$, we find that $C_{\Gamma_0(N)}(E_{M,L,\chi})\bigcap\sum_N$ must be annihilated by $p$. Thus $(C_{\Gamma_0(N)}(E_{M,L,\chi})\bigcap\sum_N)\otimes_{\BZ}\BZ[1/6N]$ is always zero, so that $C_{\Gamma_0(N)}(E_{M,L,\chi}\otimes_{\BZ}\BZ[1/6N]$ is isomorphic to $C^{(s)}_{\Gamma_0(N)}(E_{M,L,\chi})\otimes_{\BZ}\BZ[1/6N]$. Since
\begin{align*}
  \CR_{\Gamma_1(N)}(E_{M,L,\chi})\otimes_{\BZ}\BZ[\frac{1}{6N}]&\subseteq  \CR_{\Gamma_1(N)}(E_{M,L,\chi})\otimes_{\BZ}\BZ[\frac{1}{6N}]\\
&=\fa_{M,L,\chi},
\end{align*}
it follows that
\begin{align*}
  C_{\Gamma_0(N)}(E_{M,L,\chi})\otimes_{\BZ}\BZ[{1}/{6N}]&\simeq C^{(s)}_{\Gamma_0(N)}(E_{M,L,\chi})\otimes_{\BZ}\BZ[1/6N]\\
  &\simeq\Hom_{\BZ}(A^{(s)}_{\Gamma_0(N)}(E_{M,L,\chi}),{\BQ}/{\BZ})\otimes_{\BZ}\BZ[{1}/{6N}]\\
  &\simeq\Hom_{\BZ[1/6N]}(\BZ[{1}/{6N},\chi]/\fa_{M,L,\chi},{\BQ}/{\BZ[{1}/{6N}]})\\
  &\simeq\frac{\Hom_{\BZ[1/6N]}(\fa_{M,L,\chi},{\BZ[{1}/{6N}]})}{\Hom_{\BZ[1/6N]}(\BZ[1/6N,\chi],{\BZ[{1}/{6N}]})}\\
  &\simeq\BZ[{1}/{6N},\chi]/\fa_{M,L,\chi},
\end{align*}
which completes the proof of the assertion.
\end{proof}

\begin{prop}\label{index}
Notations ate as above, then $\BT_0(N)[\chi]/I^{(N)}_{M,L,\chi}$ is finite. Moreover, the action of $\BT_0(N)[\chi]$ on $C_{\Gamma_0(N)}(E_{M,L,\chi})$ induces an isomorphism
\begin{align*}
  \BT_0(N)[\chi,\frac{1}{6N}]/I^{(N)}_{M,L,\chi}\simeq C_{\Gamma_0(N)}(E_{M,L,\chi})\otimes\BZ[\chi,\frac{1}{6N}].
\end{align*}
\end{prop}
\begin{proof}
  We first prove the finiteness of $\BT_0(N)[\chi]/I^{(N)}_{M,L,\chi}$. It is clear that the inclusion $\BZ[\chi]\subset\BT_0(N)[\chi]$ induces a surjection $\varphi:\BZ[\chi]\twoheadrightarrow\BT_0(N)[\chi]/I^{(N)}_{M,L,\chi}$. Suppose $\ker({\varphi})=0$ so that $\varphi$ is an isomorphism, then the composition $\BT_0(N)\rightarrow\BT_0(N)[\chi]/I^{(N)}_{M,L,\chi}\simeq\BZ[\chi]\hookrightarrow\BC$ corresponds to a normalized cuspidal eigenform $g=\sum^{\infty}_{n=1}a_n(g)\cdot\fq^n\in S_2(\Gamma_0(N),\BC)$ such that $a_\ell(g)=\chi(\ell)^{-1}+\ell\cdot\chi(\ell)$ for any prime $\ell\nmid N$. But this contradicts to the Ramanujan bound, so $\ker(\varphi)$ must be a non-zero ideal in $\BZ[\chi]$, which proves the assertion.

  Let $\fb_{M,L,\chi}=\ker(\varphi)$ so that $\BZ[\chi]/\fb_{M,L,\chi}\simeq\BT_0(N)[\chi]/I^{(N)}_{M,L,\chi}$. From Theorem 3.2.4 of \cite{St}, $C_{\Gamma_0(N)}(E_{M,L,\chi})$ is s cyclic $\BT_0(N)[\chi]/I^{(N)}_{M,L,\chi}$-module, thus we find by Corollary~\ref{4.7} that
  \[\fb_{M,L,\chi}\otimes_{\BZ}\BZ[{\frac{1}{6N}}]\subseteq\fa_{M,L,\chi}.\]
  On the other hand, the composition $\BT_0(N)\rightarrow\BT_0(N)[\chi]/I^{(N)}_{M,L,\chi}\rightarrow\BZ[\chi,1/6N]/\fb_{M,L,\chi}$ gives a normalized eigenform $\theta\in S^B_2(\Gamma_0(N),\BZ[\chi,1/6N]/\fb_{M,L,\chi})$. Since $E_{M,L,\chi}\pmod{\fb_{M,L,\chi}}\in M^B_2(\Gamma_0(N),\BZ[\chi,1/6N]/\fb_{M,L,\chi})$ has the same eigenvalues with $\theta$, it follows from the $\fq$-expansion principle (see Proposition 1.2.10 of \cite{Oh}) that $\theta=E_{M,L,\chi}\pmod{\fb_{M,L,\chi}}$. In particular, we find that $E_{M,L,\chi}$ is a cusp form when modulo $\fb_{M,L,\chi}$, so all of its constant terms must be contained in $\fb_{M,l,\chi}$. It follows that
  \[\fa_{M,L,\chi}\subseteq\fb_{M,L,\chi}\otimes_{\BZ}\BZ[{\frac{1}{6N}}],\]
  which completes the proof of the proposition.
\end{proof}

\textbf{Proof of Theorem~\ref{main}: }It is clear that (1) implies (2). To show that (2) implies (3), let $V=J_0(N)[\fm]$ and suppose that $\fm$ is not Eisenstein. Then $\rho_\fm$ is irreducible and we find by Theorem 5.2,(b) of \cite{Ri2} that $V$ is of dimension-$2$ over ${\kappa(\fm)}$, which therefore provides a model for $\rho_\fm$. Since $0\neq\CJ_N[\fm]\subseteq V$, we have $V=\CJ_N[\fm]$. In particular, $\rho_\fm$ factors through the abelian group $\Gal(\BQ_N/\BQ)$. However, since $\fm\nmid2$ and $\rho_\fm$ is odd, $\rho_\fm$ is geometric irreducible and hence can not factors through an abelian group. This contradiction shows that $\fm$ must be Eisenstein.

It remains to prove that (3) implies (1). So let $\fm$ be a maximal Eisenstein ideal in $\BT_0(N)$. Then it is enough to show that $(\CC_N)_{\fm}$ is non-zero, where $\BT_0(N)_\fm$ is the completion of $\BT_0(N)$ at $\fm$. Let the notations be the same as in Eqs.~(\ref{eq4.5})-(\ref{eq4.6}). In particular, there is an  Eisenstein series $E_{M,L,\chi}$ associated to $\fm$. Since $\chi$ is the Teichmuller lifting of $\bar{\chi}$, the values of $\chi$ are all contained in $\BT_0(N)$ so that we have a commutative diagram
\[\xymatrix{
  \BT_0(N)[\chi] \ar[dr] \ar[r]
                & \BT_0(N)_{\fm} \ar[d]  \\
                & \kappa(\fm)             }
\]
Since the generators of $I^{(N)}_{M,L,\chi}$ are all mapped to zero in $\kappa(\fm)$, it follows that $(\BT_0(N)[\chi]/I^{(N)}_{M,L,\chi})_\fm\neq0$. Therefore we find by Proposition~\ref{index} that
\[(\CC_N)_\fm\supseteq(C_{\Gamma_0(N)}(E_{M,L,\chi}))_{\fm}\neq0,\]
which completes the proof.

\end{document}